\documentclass[a4paper,11pt]{article}
\usepackage{graphicx,amssymb,amstext,amsmath}

\usepackage[latin1]{inputenc}
\usepackage{amsthm}
\usepackage{dsfont}
\usepackage{amsmath}
\usepackage{amsfonts}
\usepackage{amssymb}
\usepackage{mathrsfs}
\usepackage{graphicx}
\usepackage{cancel}
\usepackage{graphicx,color}
\usepackage[all]{xy}
\newtheorem{Defi}{Definition}[section]
\newtheorem{R}{Remark}[section]

\newtheorem{T}{Theorem}[section]

\newtheorem{C}{Corollary}[section]

\newtheorem{lemma}{Lemma}[section]

\newcommand{\ds}{\displaystyle}
\newcommand{\B}{\mathcal{B}}

\newcommand{\re}{\mathbb{R}}

\title{\textbf{Riemannian manifolds with Anosov geodesic flow do not have conjugate points}}
\author{\'Italo  Melo  and Sergio Roma\~na\footnote{Partially supported by program   CAPES-PrInt-BR, process  N$^o$ 88887.311615/2018-00, Bolsa
Jovem Cientista do Nosso Estado No. E-26/201.432/2022 - Brazil, NNSFC 12071202, and NNSFC 12161141002 from China. The second author thanks the Department of Mathematics of the SUSTech- China for its hospitality during the execution of this work.} }  
\date{}

\begin{document}
\maketitle
%%%%%%%%%%%%%%%%%%%%%%%%%%%%%%%%%%%%%%

\begin{abstract}
This paper establishes a significant result concerning the absence of conjugate points in certain complete Riemannian manifolds. Specifically, we demonstrate that any complete non-compact manifold with curvature bounded below and an Anosov geodesic flow does not possess conjugate points. This resolves an open problem left by R. Ma\~n\'e  in \cite{man:87} and subsequently highlighted by \cite{Knieper}.
%This paper establishes a significant result concerning the absence of conjugate points in certain complete Riemannian manifolds. Specifically, we demonstrate that any complete non-compact manifold with curvature bounded below and an Anosov geodesic flow does not possess conjugate points. %Moreover, we extend this result to encompass complete Riemannian manifolds of dimension greater than $2$, characterized by bounded sectional curvature and Anosov geodesic flow, affirming the absence of conjugate points in such contexts. 
%Notably, our findings address an unresolved query outlined in \cite{Knieper}, which pertains to conjecture proposed by Ma\~n\'e in \cite{man:87} concerning the generalization of Klingenberg's theorem to non-compact manifolds.
\end{abstract}
 \section{Introduction}
%%%%%%%%%%%%%%%%%%%%%%%%%%%%%%%%%%%%%%
In \cite{Ano:69}, Anosov demonstrated that geodesic flows on compact manifolds with negative curvature exhibit chaotic dynamical behavior, leading to what are known as uniformly hyperbolic systems or simply ``Anosov" systems. Anosov's argument extends to non-compact manifolds with negatively pinched curvature, where the geodesic flow remains Anosov (cf. \cite{Knieper}).

In \cite{kli:74}, Klingenberg showed that compact manifolds with Anosov geodesic flows share several crucial properties with negatively curved manifolds. These properties include the absence of conjugate points, ergodicity of the geodesic flow, dense periodic orbits, exponential growth of the fundamental group, and zero index for every closed geodesic. While some of these results do not hold for non-compact manifolds, Ma\~n\'e's  seminal paper \cite{man:87}, using the Maslov Index, established that if the geodesic flow admits a continuous invariant Lagrangian subbundle, then there are no conjugate points. Consequently, manifolds of finite volume and  Anosov geodesic flow do not have conjugate points due to the continuity and invariance of stable and unstable bundles.

Ma\~n\'e's assertion that complete non-compact manifolds of infinite volume with bounded curvature below do not possess conjugate points under Anosov geodesic flow was made in the same paper. However, as noted in \cite[pp. 475-476]{Knieper}, there is an error in Proposition II.2 of Ma\~n\'e's proof (also discussed in \cite{Kni:18}). Subsequently, the following conjecture emerged: \\
\ \\
\noindent \textbf{Conjecture:} If M is a non-compact complete Riemannian manifold with curvature bounded below and an Anosov geodesic flow, then M has no conjugate points.

\ \\
%{
%\noindent \textbf{Conjecture:} If $M$ is a non-compact complete Riemannian manifold with curvature bounded below whose geodesic flow is Anosov, then $M$ has no conjugate points.}\\
%\ \\

This problem has garnered recent attention, especially following the noteworthy work by G. Knieper (cf. \cite{Kni:18}). Knieper addressed the conjecture by introducing three additional geometric conditions. However, it's important to highlight that while these conditions were substantial, they were considered quite stringent and did not offer complete proof of the conjecture.

The main objective of this study is to address this conjecture, aiming to prove it in dimension two without relying on any additional geometric conditions. Furthermore, we extend our analysis to any dimension, provided that the curvature remains bounded. Specifically, we establish the following theorems:
%This problem has been studied recently. The most recent work to solve this problem is due to G. Knieper (cf. \cite{Kni:18}), who solved the problem by assuming three additional geometric conditions, which are, no strong recurrence, the existence of a compact set where all possible {conjugate points} only appear in this compact set and the existence of a geodesic without conjugate points.

%The main goal of this work is to prove this conjecture in dimension two without assuming any additional geometric conditions. Moreover, we also prove the conjecture in any dimension, when the curvature is bounded. More specifically, we prove the following theorems:

%\textcolor{red}{Despite the mistake, in Ma\~n\'e's proof there exists certain statements remain valid and will be useful in our proof. In [Kni18], Knieper provide a positive answer to the above question with three additional conditions, among them, the existence of a geodesic without conjugate points and the existence of a compact set $K \subset M$ and a constant $\sigma > 0$ such that for all $p \in M \backslash K$ and all geodesics $\gamma$ with $\gamma(0) = p$ the segment $c : [-1, \sigma] \to M$ has no conjugate points.}

%\textcolor{red}{In this paper we provide a positive answer to the above question without additional hypothesis, we will prove the following result.}

\begin{T}\label{main2}
Suppose $M$ is a non-compact two-dimensional manifold with sectional curvature bounded below. If the geodesic flow of $M$ is Anosov, then $M$ has no conjugate points.
\end{T}

For our second result, we define a manifold $M$ to have bounded sectional curvature if there exist two positive constants $k$ and $b$ such that

$$-k^2\leq K_{M}\leq b^2.$$ 

\begin{T}\label{main1}
Let $M$ be a non-compact complete Riemannian manifold with bounded sectional curvature. If the geodesic flow of $M$ is Anosov, then $M$ has no conjugate points.
\end{T}

In Remark \ref{RNEW1}, we show that Theorem \ref{main1} still holds under a weaker condition on the sectional curvature.

%\begin{T}\label{main2}
%Let $M$ be a non-compact two-dimensional manifold with sectional bounded below. If the geodesic flow of $M$ is Anosov, then $M$  has no conjugate points.
%\end{T}

%For our second result, given a manifold $M$ we said that $M$ has sectional curvature $K_M$ bounded if there are two positive constants $k$ and $b$ such that 

%$$-k^2\leq K_{M}\leq b^2.$$ 

%\begin{T}\label{main1}
%Let $M$ be a non-compact complete Riemannian manifold with bounded sectional. If the geodesic flow of $M$ is Anosov, then $M$  has no conjugate points.
%\end{T}

%In Remark \ref{RNEW1}, we have a weaker condition on the sectional curvature which Theorem \ref{main1} is also valid. 
It is worth noting that in our pursuit of results, we have not expanded upon any techniques utilized by Knieper in \cite{Kni:18}. Deliberately, we have avoided employing Ma\~n\'e's methodologies and have refrained from modifying his original proof. Instead, we have embarked on an extensive exploration of the index formula, which has revealed profound insights into the precise timing of the emergence of conjugate points.\\
\indent In the original result by Klingenberg \cite{kli:74} for the compact case and by Ma\~n\'e for finite volume \cite{man:87}, the recurrence of the geodesic flow was a crucial property in establishing the absence of conjugate points for Anosov geometries. Our result is very surprising, as we do not require any condition on the recurrence of the geodesic flow to obtain geometries without conjugate points.

\paragraph{Paper Structure:} This paper follows the subsequent outline: In Section \ref{Not_Pre}, we delve into the fundamental concept of geodesic flow and symplectic geometry concerning the unitary tangent bundle. Section \ref{Classic results} presents pivotal classical findings essential for constructing our arguments. The longest and most crucial section, Section \ref{Intersections}, explores profound properties of the points where the vertical and stable (unstable) bundles intersect non-trivially, intricately linked with Section \ref{Classic results}. Finally, in Section \ref{Main-Results}, we establish the proofs for Theorem \ref{main1} and Theorem \ref{main2}.

%This paper is organized as follows: In Section \ref{Not_Pre} we comment on the basic notion of geodesic flow and symplectic geometry for the unitary tangent bundle. In Section 
%\ref{Classic results} we present some classical results necessary to build our arguments. In Section \ref{Intersections} we study deep properties of the set of points where the vertical and stable (unstable) bundles has a non-trivial intersection, this section is the longest and the most important of the paper and is connected with the section \ref{Classic results}. Finally, in Section \ref{Main-Results} we prove the Theorem \ref{main1} and the Theorem \ref{main2}.

%\textcolor{red}{ In this way, the original Ma��'s claim in \cite{man:87} is true.}
\section{Notation and Basic Concepts}\label{Not_Pre}

Throughout the rest of this paper, $M=(M, \langle \, , \, \rangle)$ will denote a complete Riemannian manifold without boundary and dimension $m\geq 2$.  We denote by $TM$ the tangent bundle and $SM$ its unit tangent bundle.

%In this section we present the notations and some preliminaries results that will be used in the sequel. Throughout the rest of the
%paper, $(M,g)$ denotes a complete Riemannian manifold without boundary with dimension $n \geq 2$ and $m$ the Liouville measure of the unitary tangent bundle $SM$.

%$\nu_g$ denotes the Riemannian measure or volume measure associated to Riemannian manifold $M$.
\subsection{Geodesic flow}\label{GeoFlow}
%\subsubsection{}
For a given $\theta=(p,v) \in TM$, {we define} $\gamma_{_{\theta}}(t)$ {as} the unique geodesic with initial conditions $\gamma_{_{\theta}}(0)=p$ and 
$\gamma_{_{\theta}}'(0)=v$. For a given $t\in \mathbb{R}$, let $\phi^t:TM \to TM$ be the diffeomorphism given by 
$\phi^{{t}}(\theta)=(\gamma_{_{\theta}}(t),\gamma_{_{\theta}}'(t))$. Recall that this family is a flow (called the \textit{geodesic flow}) in the sense that  $\phi^{t+s}=\phi^{t}\circ \phi^{s}$ for all $t,s\in \mathbb{R}$. 

Let $V:=\operatorname{ker}\, D\pi$ be the \textit{vertical} {subbundle} of $T(TM)$ (tangent bundle of $TM$),  where $\pi\colon TM \to M$ is the canonical projection. \\
Let $K\colon T(TM)\to TM$ be the Levi-Civita connection map of $M$ and $H:=\operatorname{ker} K$ be the horizontal subbundle. The map {$K$} is defined as follows: Let $\xi \in T_{_{\theta}}TM$ and $z:(-\epsilon,\epsilon) \rightarrow TM$  be a curve adapted to $\xi$, \emph{i.e.},  $z(0) = \theta$ and $z'(0) = \xi$, where $z(t) = (\alpha(t),Z(t))$, then
$$K_{_{\theta}}(\xi)=\nabla_{\frac{\partial}{\partial\,t}}Z(t)\Big|_{t=0}.$$
%\begin{definition}

%\textcolor{green}{Observe that for all $\theta=(p,v)$, $D\pi$ defines a linear isomorphism from $H(\theta)$ to $T_{p}M$ and $\alpha$ defines a linear isomorphism from $V(\theta)$ to $T_{p}M$. Therefore, $T_{_{\theta}}TM = H(\theta) \oplus V(\theta)$.} 

{For each $\theta$, the maps $D_{\theta}\pi|_{H(\theta)}: H(\theta) \rightarrow T_pM$ and $K_{_{\theta}}|_{V(\theta)}:V(\theta) \rightarrow T_pM$ are linear isomorphisms. Furthermore, $T_{_{\theta}}TM = H(\theta) \oplus V(\theta)$ and the map  $j_{_{\theta}}:T_{_{\theta}}TM \rightarrow T_pM \times T_pM$ given by 
	$$ j_{_{\theta}}(\xi) = (D_{_{\theta}}\pi(\xi),K_{_{\theta}}(\xi)),        $$
	is a linear isomorphism. }

{Using the decomposition $T_{_{\theta}}TM = H(\theta) \oplus V(\theta)$, we can identify a vector $\xi \in T_{_{\theta}} TM$ with the pair of vectors $D_{_{\theta}}\pi(\xi)$ and $K_{_{\theta}}(\xi)$ in $T_{p}M$. The \textit{Sasaki metric} is a metric that makes $H(\theta)$ and $V(\theta)$ orthogonal and is given by}
$$ g_{_{\theta}}^{S}(\xi,\eta) = \langle D_{_{\theta}}\pi(\xi), D_{_{\theta}}\pi(\eta)  \rangle  +  \langle K_{_{\theta}}(\xi),    K_{_{\theta}}(\eta)   \rangle. $$
Observe that $SM$ is invariant by $\phi^t$, thus,  from now on, we consider $\phi^t$ restricted to $SM$ and $SM$ endowed with the  Sasaki metric.\\
%Observe that $SM$ is invariant by $\phi^t$, thus,  from now on, we $\phi^t$ restricted to $SM$$ and $SM$ endowed with the consider the Sasaki metric restricted to the unit tangent bundle $SM$. 
The types of geodesic {flows} that we discuss in this paper are the Anosov geodesic flows, whose definition follows below.\\
We say that the geodesic flow $\phi^t:SM \rightarrow SM$ is Anosov (with {respect to} the Sasaki metric on $SM$) if $T(SM)$ has a splitting $T(SM) = E^s \oplus \langle G \rangle \oplus E^u $ such that 
\begin{eqnarray*}
	d\phi^t_{\theta} (E^s(\theta)) &=& E^s(\phi^t(\theta)),\\
	%d\phi_t(x) (X_c^{*}(x)) &=& X_c^{*}(\phi_t(x)),\\
	d\phi^t_{\theta} (E^u(\theta)) &=& E^u(\phi^t(\theta)),\\
	||d\phi^t_{\theta}\big{|}_{E^s}|| &\leq& C \lambda^{t},\\
	||d\phi^{-t}_{\theta}\big{|}_{E^u}|| &\leq& C \lambda^{t},
	%	a\exp(-tb)\cdot ||\xi|| &\leq& ||d\phi_t(x)(\xi)||  \leq a\exp(tb) \cdot ||\xi||, \forall \xi \in X_c^{*},
\end{eqnarray*}
for all $t\geq 0$ with $ C > 0$ and $0 < \lambda <1$,  where $G$ is the vector field derivative of the geodesic flow.

\subsection{Jacobi fields and the differential of
the geodesic flow}
The Jacobi fields are important geometrical tools to understand the behavior of the differential of the geodesic flow. A vector field $J$ along $\gamma_{_{\theta}}$ is called the \emph{Jacobi field} if it satisfies the equation 
\begin{equation*}
J'' + R(\gamma'_{_{\theta}},J)\gamma'_{_{\theta}} = 0,
\end{equation*}
where $R$ is the Riemann curvature tensor of $M$ and $``\,'\,"$ denotes the covariant derivative along $\gamma_{_{\theta}}$. \\
For  $\theta=(p,v)$ and  $\xi=(w_{1},w_2) \in T_{_{\theta}}SM$, (the horizontal and vertical decomposition) with $w_1, w_2 \in T_{p} M$ and $\langle v, w_2 \rangle =0$. It is known that
\begin{equation}\label{eq4}
d\phi^{t}_{\theta}(\xi) = (J_{_{\xi}}(t), J'_{\xi}(t)),
\end{equation}  
where $J_{_{\xi}}$ denotes the unique Jacobi vector field along $\gamma_{_{\theta}}$ such that $J_{_{\xi}}(0) = w_1$ and $J'_{\xi}(0) = w_2$ (see \cite{P}). The equation (\ref{eq4}) enables us to assert that the investigation of the dynamics of the geodesic flow revolves around Jacobi fields.
 
%\subsubsection{Conjugate points}
Another crucial concept, intimately tied to this work, is that of conjugate points.

\begin{Defi}
Two points $p$ and $q$ on a Riemannian manifold to be \emph{conjugate} if there exists a geodesic $\gamma$ connecting $p$ and $q$, along which there exists a non-zero Jacobi field that vanishes at both $p$ and $q$. When no two points in $M$ are conjugate, we say that the manifold $M$ \emph{has no conjugate points}.
\end{Defi}
%Let $p$ and $q$ points on a Riemannian manifold, we say that $p$ and $q$ are \emph{{conjugate}} if there is a geodesic $\gamma$  that connects $p$ and $q$ and a non-zero Jacobi field along $\gamma$  that vanishes at $p$ and $q$. When neither two points in $M$ are {conjugate}, we say the manifold $M$ \emph{has no conjugate points}.
Geometries without conjugate points are fundamental in mathematics, particularly in differential geometry and dynamical systems. They provide insights into the global structure of Riemannian manifolds, influencing the behavior of geodesics. The absence of conjugate points allows geodesics to extend indefinitely without intersecting themselves in the universal covering impacting various mathematical analyses and theorems related to geodesic flows, curvature, and topology. Additionally, in these geometries, the universal cover is an Euclidean space, further illustrating their significance and utility in mathematical investigations.\\
As mentioned in the introduction, metrics whose geodesic flow is Anosov, in the compact case \cite{kli:74} and in the case of finite volume \cite{man:87}, exhibit a geometry without conjugate points.

\subsection{Symplectic geometry}
The Riemannian {geometry} provides to unitary tangent bundle a natural symplectic structure using the horizontal and vertical decomposition given in Subsection \ref{GeoFlow}.\\
We define a symplectic form $\Omega$ {and  one-form $\beta$}  given by
$$   \Omega_{\theta}(\xi, \eta) =  \langle   D_{_{\theta}}\pi(\xi), K_{_{\theta}}(\eta) \rangle - \langle   D_{_{\theta}}\pi(\eta), K_{_{\theta}}(\xi) \rangle,                                  $$

%\textcolor{green}{This decomposition provides the Sasaki metric of $TM$ 
%is defined by $g_{_{\theta}}^{S}((X,Y),(X,Y)):=g(X,X) + g(Y,Y).$}

%\textcolor{red}{The geodesic flow preserves the volume measure generated by this Riemannian metric in $SM$.} We denote by $G$ the vector field derivative of the geodesic vector flow, in the above identification $G(\theta) = (v, 0)$.}\\ Now consider the one-form $\beta$ {of} $SM$ defined {as}
$$\beta_{_{\theta}}(\xi) = g_{_{\theta}}^S(\xi,G(\theta)) = \langle D_{_{\theta}}\pi(\xi),v \rangle_p.$$
Observe that $\operatorname{ker}\, \beta_{_{\theta}}\supset V(\theta) \cap T_{_{\theta}} SM $.  It is possible to prove that a vector $\xi \in T_{_{\theta}}TM$ lies in $T_{_{\theta}}SM$ 
with $\theta=(p,v)$ if and only if $\langle K_{_{\theta}}(\xi) , v\rangle = 0$. Furthermore, 
$\beta$ is an invariant contact form by the geodesic flow whose Reeb vector field is the geodesic vector field $G$. 

The subbundle $S=\operatorname{ker} \beta$ is the orthogonal complement of the subspace spanned by $G$. Since $\beta$ is invariant by the geodesic flow, then the subbundle $S$ is invariant by {$\phi^{t}$}, \emph{i.e.}, $d\phi^t(S(\theta)) = S(\phi^{t}(\theta))$ for all $\theta\in SM$ and all $t\in\mathbb{R}$.

It is known that the restriction of $\Omega_{\theta}$ to $S(\theta)$ is nondegenerate and invariant by $\phi^{t}$ (see \cite{P} for more details).

%\textcolor{blue}{The types of geodesic flows that we discuss in this paper are the Anosov geodesic  flows, whose definition follows below.}
%\begin{Defi} We say that the geodesic flow $\phi^t:SM \rightarrow SM$ is Anosov if $SM$ have a splitting $SM = E^s \oplus \langle G \rangle \oplus E^u $ such that 
%\begin{eqnarray*}
%	d\phi^t_{\theta} (E^s(\theta)) &=& E^s(\phi^t(\theta)),\\
	%d\phi_t(x) (X_c^{*}(x)) &=& X_c^{*}(\phi_t(x)),\\
%	d\phi^t_{\theta} (E^u(\theta)) &=& E^u(\phi^t(\theta)),\\
%	||d\phi^t_{\theta}\big{|}_{E^s}|| &\leq& C \lambda^{t},\\
%	||d\phi^{-t}_{\theta}\big{|}_{E^u}|| &\leq& C \lambda^{t},\\
	%	a\exp(-tb)\cdot ||\xi|| &\leq& ||d\phi_t(x)(\xi)||  \leq a\exp(tb) \cdot ||\xi||, \forall \xi \in X_c^{*},
%\end{eqnarray*}
%for all $t\geq 0$ with $ C > 0$ and $0 < \lambda <1$.
%\end{Defi}
%For more details see \cite{P}.
\subsection{Graphs and Riccati equation}\label{NCP}
For  $\theta = (p, v) \in SM$, let $N(\theta):= \{w \in T_{x}M : \langle w, v\rangle = 0\}$. By the identification of Subsection \ref{GeoFlow}, we can write $S(\theta):=\operatorname{ker} \, \beta = N(\theta)\times N(\theta)$, $V(\theta) \cap S(\theta) =  \{0\} \times N(\theta) $ and $H(\theta) \cap S(\theta) = N(\theta) \times  \{0\} $.

\begin{Defi}
	A subspace $E \subset S(\theta)$ with $\dim E = m-1$ is said to be Lagrangian if  $\Omega_{\theta}(\xi, \eta) = 0$ for any $\xi, \eta \in S(\theta)$.
\end{Defi}
{The Lagrangian {subbundles} play an important role in this paper (see Lemma \ref{Discrete Intersection}), since in the Anosov case, it is known that  for each $\theta\in SM$, the {subspace} $E^{s}({\theta})$ and {the subspace} $E^{u}(\theta)$ are Lagrangian (cf. \cite{man:87} and \cite{P}).}

 Observe that  if $E \subset S(\theta)$ is a subspace {with}
$\dim E = m-1$ and $E \cap V(\theta) = \{0\}$, then 
$E \cap (H(\theta) \cap S(\theta))^{\perp} = \{0\}$. Hence, there exists a unique linear map $T: H(\theta) \cap S(\theta) \to  V(\theta) \cap S(\theta)$ such that $E$ is the graph of  $T$. In other words, there exists a unique linear map $T: N(\theta) \to N(\theta)$ such that $E = \{(v,Tv) : v \in N(\theta)\}$. Furthermore, the linear map $T$ is symmetric if and only if $E$ is Lagrangian.

%\textcolor{red}{Now we describe a useful method of L. Green (cf. \cite{Gre:58}). Let $\gamma_{\theta}$ be a geodesic, and consider $V_1,\dots,V_n$ a system of parallel orthonormal vector fields  along $\gamma_{\theta}$ with \linebreak $V_n(t) = \gamma'_{\theta}(t)$.
%If $Z(t)$ is a perpendicular vector field along $\gamma_{\theta}(t)$, we can write $$Z(t) =\displaystyle \sum_{i=1}^{n-1} y_{i}(t)V_i(t).$$ 
%The covariant derivative $Z'(s)$ is identified with the curve $\alpha'(s) = (y'_{1}(s),...,y'_{n-1}(s))$. Conversely, any curve in $\mathbb{R}^{n-1}$ can be identified with a perpendicular vector field on $\gamma_{\theta}(t)$.
%For each $t \in \mathbb{R}$, consider the symmetric matrix $R(t) = (R_{i,j}(t))$, where $1 \leq i,j \leq n-1$, $R_{i,j} = \langle R(\gamma'_{\theta}(t), V_i(t))\gamma'_{\theta}(t),  V_j(t)) \rangle$ and $R$ is the curvature tensor of $M$. Consider the $(n-1) \times (n-1)$ matrix Jacobi equation
%\begin{equation}\label{Jacobi}
%Y''(t) + R(t)Y(t) = 0.
%\end{equation}
%If $Y(t)$ is solution of (\ref{Jacobi}) then for each $x \in \mathbb{R}^{n-1}$, the curve $ \beta(t) = Y(t)x$ corresponds to a Jacobi perpendicular vector on $\gamma_{\theta}(t)$.} \\
Let $E$ be an invariant Lagrangian subbundle, i.e, for every $\theta \in SM$, $E(\theta) \subset S(\theta)$ is a Lagrangian subspace and $d\phi^t(E(\theta)) = E(\phi^{t}(\theta))$, for all $t \in \mathbb{R}$. Suppose that $E(\phi^{t}(\theta)) \cap V(\phi^{t}(\theta)) = \{0\}$ for every $t \in (-\delta, \delta)$. We can to write $E(\phi^{t}(\theta)) = {\rm{graph}}\, U(t)$ for all $t \in (-\delta, \delta)$, with $U(t): N(\phi^{t}(\theta)) \to N(\phi^{t}(\theta))$ which satisfies the Ricatti equation 
\begin{equation*}\label{R}
  U'(t)+ U^2(t) + R(t) = 0,
\end{equation*}
for more details see  \cite{Gre:58},  \cite{Ebe:73} or \cite[Section 2]{Italo-Sergio-2}.
%\textcolor{red}{  for more details see \cite{P}.
% from the method of L. Green (cf. \cite{Gre:58},  \cite{Ebe:73} ) we can consider $U(t)\colon \mathbb{R}^{n-1} \to \mathbb{R}^{n-1}$. The family of operators $U(t)$ satisfies the Riccati equation} 
  \section{Classical results }\label{Classic results}
  In this section, we introduce key results and concepts crucial for proving our main theorem. The first lemma, attributed to Ma\~n\'e (cf. \cite{man:87} and \cite{P} for further details), reveals a ``twist property" between the vertical subbundle and a Lagrangian subbundle along orbits. Specifically,
 
\begin{lemma}{\emph{\cite[Lemma III.2]{man:87}}}\label{Discrete Intersection}
	If $\theta \in SM$ and $E \subset S(\theta)$ is a Lagrangian subspace, then the set of \, $t \in \mathbb{R}$ such that $	d\phi^t_{\theta} (E) \cap V (\phi^t({\theta}) ) \neq \{0\}$ is discrete.
\end{lemma}
%\begin{lemma}{\emph{\cite[Lemma III.2]{man:87}}}
%	If $\theta \in SM$ and $E \subset S(\theta)$ is a Lagrangian subspace, then the set of $t \in \mathbb{R}$ such that $	d\phi^t_{\theta} (E) \cap V (\phi^t_{\theta} ) \neq \{0\}$ is discrete.
%\end{lemma}

Recall that a geodesic arc $\gamma:[a,b] \to M$ is devoid of conjugate points if, for any Jacobi vector field with $J(c) = 0$ and $J'(c) \neq 0$, it holds that $J(t) \neq 0$ for all $t \in [a,b]$ except $t = c$. In the subsequent theorem, Ma\~n\'e established a correlation between the trivial intersection of the vertical subbundle and a Lagrangian subbundle and the absence of conjugate points.

%Remember that a geodesic arc $\gamma:[a,b] \to M$ does not contain conjugate points if for any Jacobi vector field with $J(c) = 0$ and $J'(c) \neq 0$ follows that $J(t) \neq 0$ for any $t \in [a,b]$ with $t \neq c$. In the next result, Ma\~n\'e related the trivial intersection between the vertical subbundle and a Lagrangian subbundle and the {nonexistence} of conjugate points.
\begin{lemma}\emph{\cite[Proposition II.1]{man:87}}\label{Mane Lemma}
Let $M$ be a Riemannian manifold and $\gamma \colon [0,a] \to M$ a geodesic arc. If there exists a Lagrangian subspace $E\subset S(\gamma(0), \gamma'(0))$ such that $V(\gamma(t), \gamma'(t))\cap d\phi^{t}(E)=\{0\}$ for all $0\leq t\leq a$, then the geodesic arc $\gamma$ does not contain conjugates points. 
\end{lemma}

%\begin{Pro}{\emph{\cite[Proposition II.1]{man:87}}} Let $M$ be a Riemannian manifold and \linebreak $\gamma:[0,a] \to M$ a geodesic arc. If $E \subset S(\theta)$ is a Lagrangian subspace such that \linebreak $d\phi^t_{\theta} (E) \cap V (\phi^t_{\theta} ) = \{0\}$ for all $0 \leq t \leq a$ then the geodesic arc $\gamma$ doesn't contain conjugate points.
%\end{Pro}
For an alternative proof of the above result, see  \cite{Knieper}. In the Anosov case, the stable ($E^s$) and unstable ($E^u$) subbundles  are invariant, continuous, and Lagrangian. In this way, the above results are valid for $E^{s}$ and $E^{u}$.

As an immediate consequence of Lemma \ref{Mane Lemma}  we obtain the following result.
\begin{C}\label{C - Transfer Prop}
Let $J$ be a nonzero Jacobi field along $\gamma_{\theta}$ such that  $J(a)=J(b)=0$ for $a<b$, then there are $c, d \in [a,b]$, a nonzero stable Jacobi filed $J^s$ and non zero unstable Jacobi field  $J^u$ such that $J^{s}(c)= 0$ {and} $J^u(d)=0$.
\end{C}

%\textcolor{red}{In the next result, Eberlein in \cite{Ebe:73} studied the invariant Lagrangian subbundle solutions of (\ref{R}) defined for all $s>0$. 
%\begin{lemma}{\emph{\cite[Lemma 2.8]{Ebe:73}}}
%For any integer $n \geq 2$ consider the $(n-1) \times (n-1)$ matrix Riccati equation
%$$   U'(s) + U(s)^2 + R(s) =0,                           $$
%where $R(s)$ is a symmetric matrix such that $\langle R(s)x, x \rangle > -k^2$ for some $k > 0$, all unit vectors $x \in \mathbb{R}^{n-1}$ and all real number $s$. If $U(s)$ is a symmetric solution of the Ricatti equation, which is defined for all $s >0$, then $|  \langle U(s)x, x \rangle         | \leq k \coth (ks)$ for all $s > 0$ and all unit vectors $x \in \mathbb{R}^{n-1}$.
%\end{lemma}}

%\textcolor{red}{\begin{lemma}[\cite{man:87}, Lemma II.3]
%Let $M$ be a complete Riemannian manifold with curvature bounded below by $-c^2$. Then there exists $A=A(c)>0$ such that if $\gamma\colon [0,a]\to M$ is a geodesic arc without conjugate points and $J$, $0\leq t\leq a$ is a perpendicular Jacobi field on $\gamma$ with $J(0)=0$, then $$\|J'(t)\|\leq A\|J(t)\|$$
%for all $1\leq t\leq a$.
%\end{lemma}}

\indent When the manifold exhibits an Anosov geodesic flow and the curvature is bounded below, the stable and unstable subbundles possess two fundamental properties. The first, established by Knieper (cf. \cite{Knieper}), demonstrates that the zeros of stable and unstable Jacobi fields correspond to conjugate points. The second property arises from Green's method (cf. Green \cite{Gre:58}), which yields a uniform upper bound for the Riccati solution in instances where stable or unstable Jacobi fields lack zeros. To elaborate further,  
\begin{lemma}{\emph{\cite[Lemma 3.5]{Knieper}}}\label{Knieper Lemma}
	Let $M$ be a Riemannian manifold with curvature bounded below. 
	If the geodesic flow is Anosov {then} there exists a constant $\sigma$ with the following property. If
	$$E^{s}(\theta)\cap V(\theta)\neq \{0\}$$ 
	then $\gamma_\theta$ has conjugate points on the interval $[- 1, \sigma]$. If $$E^{u}(\theta)\cap V(\theta)\neq \{0\}$$ then  $\gamma_\theta$ has conjugate points on the interval $[-\sigma, 1]$. 
\end{lemma}
%\textcolor{blue}{Apagar o Remark 3.1}
%\begin{R}\label{R-Knieper}
%	Actually, the proof of Lemma \ref{Knieper Lemma} showed that if $\xi^s\in E^{s}(\theta)\cap V(\theta)$ and $J^{s}$ is the stable Jacobi field associated to $\xi^s$, then $J^{s}$ has another zero on $[-1,\sigma]$. Analogous arguments to unstable case. 
%\end{R}
As $E^s$ and $E^u$ are Lagrangian, then using the notation of {Subsection} \ref{NCP}, we have % the second property is: 
\begin{lemma}\label{Closed}
Let $M$ be a Riemannian manifold with curvature bounded below by $-k^{2}$. Assume that the geodesic flow is Anosov, then if $\theta\in SM$ satisfies that 
$$E^s(\phi^{t}(\theta))\cap V(\phi^t(\theta))=\{0\} \,\,\, \text{for all} \,\, t\in \mathbb{R},$$
then $$\displaystyle \sup_{t\in \mathbb{R}}\|U^s_{\theta}(t)\|\leq k,$$
where $U^s_{\theta}(t)\colon N(\phi^t(\theta))\to N(\phi^t(\theta))$   is the symmetric linear map such that $E^{s}(\phi^{t}(\theta)) = {\rm{graph}}\, U^s_{\theta}(t)$. An analogous result holds for the unstable case.
\end{lemma}
\noindent We also need a result due to  Eberlein (cf. \cite{Ebe:73}).%also  \cite{Gre:58}).% Corollary \ref{Bounded t^{s}_{-}}, and Lemma \ref{kappa}.  %will useful to get a geodesic without conjugate points. 
\begin{lemma}{\emph{\text{\cite[Lemma 2.8]{Ebe:73}}}}\label{eberlein}
For any integer $n>2$ consider the $(n-1)\times(n-1)$ matrix
\emph{Riccati equation}

\begin{equation}\label{Ricatti}
U'(s) + U^{2}(s) + R(s)= 0,
\end{equation}
where $R(s)$ is a symmetric matrix such that $\langle R(s)x, x \rangle > -k^2$
 for some $k > 0$,
all unit vectors $x\in \mathbb{R}^{n-1}$ and all real numbers $s$. If $U(s)$ is a symmetric solution of \emph{(\ref{Ricatti})}, which is defined for all $s > 0$, then $\langle U(s)x, x \rangle < k \coth (ks)$ for
all $s > 0$ and all unit vectors $x\in \mathbb{R}^{n-1}$.
\end{lemma}

%%%%%%%%%%%%%%%%%%%%%%%%%%%%%%%%%%%%%%%%%%%%%%%%%%%%%%%%%%%%%%%%%%%%%
%%%%%%%%%%%%%%%%%%%%%%%%%%%%%%%%%%%%%%%%%%%%%%%%%%%%%%%%%%%%%%%%%%%%%%%%%%%%%%%%%

%%%%%%%%%%%%%%%%%%%%%%%%%%%%%%%%%%%%%%%%%%%%%%%%%%%%%%%%%
%%%%%%%%%%%%%%%%%%%%%%%%%%%%%%%%%%%%%%%%%%%%%%%%%%%%%%%%%%%%%%%%%%%%%%%%%%%%%%%%%%%%%%%%%%%%%%%%%%%%%%%%%%%%%%%%%%

%%%%%%%%%%%%%%%%%%%%%%%%%%%%%%%%%%%%%%%%%%%%%%%%%%%%%%%%%%%%%%%%%%%%%%

\section{The intersection between the vertical  subspace and $E^{s,u}$}\label{Intersections}
From Lemma \ref{Knieper Lemma} we need to avoid non-trivial intersection of the stable (unstable) bundle with the vertical bundle. So, we consider the following subset of $SM$

$$\B^{s,u}=\Big\lbrace\theta\in SM: V(\theta)\cap E^{s,u}(\theta)\neq \{0\}\Big \rbrace\textcolor{blue}{.}$$

%\textcolor{red}{\noindent Using the sets $\B^{s,u}$ and Lemma \ref{Mane Lemma}, {to order to prove} the Theorem \ref{main1} is {sufficient} to prove {the following theorem.}
%\begin{T}\label{main2}
%{Let $M$ be a non-compact complete Riemannian manifold with  curvature bounded below. If the geodesic flow of $M$ is Anosov}, then  $\B^{s}=\emptyset$ and $\B^{u}=\emptyset$. 
%\end{T}}
The rest of this section is devoted to finding some properties of $\B^{s,u}$. More specifically, we will prove that $\B^{s,u}=\emptyset$ (see Corollary \ref{Cor1-Main1}, which implies the Theorem \ref{main1}). % lemmas needed to prove Theorem \ref{main1}. %For this sake, we assume that both $\B^{s}$ and $\B^{u}$  are nonempty and we will construct a contradiction. 
%In the following subsection, we study some properties of the sets $\B^{s,u}$.

\subsection{Properties of sets $\B^{s}$ and $\B^u$}

\begin{lemma}\label{Le-Closed} The sets $\B^{s}, \B^{u}\subset SM$ are closed.
\end{lemma}
\begin{proof}
Let $\theta_n \in \B^{s,u}$ be, $\theta_n\to \theta$, then there is $z_n\in V(\theta_n)\cap E^{s,u}(\theta_n)$. As the spaces involved are subspaces, then we can assume that $\|z_n\|=1$. Thus, since $V(\theta_n)=\text{ker} \, d\pi_{\theta_n}$ and $E^{s,u}$ are continuous subbundles, then passing to a subsequence if necessary,   $z_n\to z\in V(\theta)\cap E^{s,u}(\theta)$, which implies that $\theta\in \B^{s,u}$. 
\end{proof}
\begin{R}\label{R1-Properties} Using the notation of \emph{Subsection} \emph{\ref{NCP}}, if $\theta\notin \B^{s,u}$, then there are unique linear maps  $T^{s,u}_{\theta}\colon {H(\theta) \cap S(\theta) \to V(\theta) \cap S(\theta) }$  such that $E^{s,u}(\theta)$ is the graph of $T^{s,u)}_{\theta}$, \emph{i.e.}, 
$$E^{s,u}(\theta)=\{(z,T^{s,u}_{\theta}(z)): z\in H(\theta)\}.$$
\end{R}
\noindent Using the horizontal and vertical coordinates from {Subsection} \ref{NCP}, the following lemma yields a local uniform control of the norm of the linear maps $T^{s,u}$, enabling us to govern the vertical coordinates through the horizontal ones. 
\begin{lemma}\label{stable-norm}
Assume that $\theta\notin \B^{s,u}$, then there is a compact neighborhood $\mathcal{U}_{\theta}^{s,u}\subset SM\setminus \B^{s,u}$ of $\theta$ and $\alpha_{s,u}(\theta)>0$ such that if $(v,w)\in E^{s,u}(z)$, then 
$$\|w\|\leq \alpha_{s,u}(\theta)\|v\| \, \, \, \text{for all}\, \, \,z \in \mathcal{U}_{\theta}^{s,u}.$$
% Analogous result for unstable case.  
\end{lemma}
\begin{proof}
We prove the stable case since the unstable case is analogous. \\
Let $\theta\notin \B^s$ be, then by Lemma \ref{Le-Closed} there is a compact neighborhood $\mathcal{U}_{\theta}^{s}\subset SM\setminus \B^s$ of $\theta$ such that  for all  $z\in \mathcal{U}_{\theta}^{s}$, we have that 
$$V(z)\cap E^{s}(z)=\{0\}.$$
Then, by Remark \ref{R1-Properties}, for each $z\in \mathcal{U}_{\theta}^{s}$ there is a unique linear map $$T^s_{z}\colon {H(z)\cap S(z)  \to V(z) \cap S(z),}$$ such that $E^s(z)$ is the graph of $T^s_{z}$. The compactness of {$\mathcal{U}_{\theta}^{s}$} and the continuity of $E^s$ allows us to state that there is $\alpha_{s}(\theta)>0$ such that $$\|T^{s}_{z}\|\leq \alpha_{s}(\theta) \, \, \,  \text{for all}\, \, \,  z\in {\mathcal{U}_{\theta}^{s}}.$$
Observe that if $(v,w)\in E^s(z)$, then $(v,w)=(v,T^s_{z}(v))$, which implies that 
$$\|w\|=\|T^{s}_{z}(v)\|\leq \alpha_{s}(\theta)\|v\|.$$
\end{proof}

In the context of $\B^{s,u}$, as a corollary of Lemma \ref{eberlein}, we have (compare with Lemma \ref{Closed}).
\begin{lemma}\label{Good t>0}  Let $M$ be a Riemannian manifold with curvature bounded below by $-k^2$, for some $k>0$, with Anosov geodesic flow.
Let $\theta\in SM$ such that $\phi^t(\theta)\notin \B^u$ for all $t \geq 0$. If $U^{u}$ is the symmetric solution of the Ricatti equation \emph{(\ref{Ricatti})} associated to the unstable bundle $E^u$ on $[0,+\infty)$, then
$$\langle U^{u}(t)x, x \rangle < k \coth (kt),$$
for all $t > 0$ and all unit vectors $x\in \mathbb{R}^{n-1}$. Analogous result for stable cases.
\end{lemma}
\begin{proof}
Note simply that if $\phi^t(\theta)\notin \B^u$ for {$t \geq 0$}, then $U^u(t)$ is defined for {$t \geq0$} and the result follows from  Lemma \ref{eberlein}.
\end{proof}
\begin{R}\label{R-Good t>0}In the same conditions of the previous lemma, observe that $\coth kt \leq 2 $ for all $t\geq \dfrac{1}{k}$ and then 
\begin{equation}\label{u-Ricatti}
\langle U^{s}(-t)x, x \rangle \leq 2k  \, \, \, \text{for all} \, \, \, t\geq \frac{1}{k} 
\end{equation}
and all unit vectors $x\in \mathbb{R}^{n-1}$. In particular, $\|U^{s}(-t)\| \leq  2k$ for all $t\geq \frac{1}{k}$.
\end{R}
Using the notation of Lemma \ref{stable-norm}, %\ref{Bounded t^{s}_{-}}, 
then an immediate consequence,  we have 
\begin{C}\label{Unif Bounded t_{+}*}
In the conditions of the previous lemma, let $\theta\in SM$ such that $\phi^{-t}(\theta)\notin \B^s$ for all $t\geq 0$. Then for all $t\geq \dfrac{1}{k}$, 
$$\alpha_{s}(\phi^{-t}(\theta))\leq 2k.$$
Analogous for the stable case. 
\end{C}
%\textcolor{blue}{Rever a prova, ajustar o valor inicial de $t$}
\begin{proof}
Assume that $t\geq \dfrac{1}{k}$ and consider $(v,w)\in E^s(\phi^{-t}(\theta))$, then by Lemma \ref{stable-norm} and Remark \ref{R-Good t>0}  we have that $\|w\|\leq 2k\|v\|$. 
Thus, $\alpha_{s}(\phi^{-t}(\theta))\leq 2k$, for all $t\geq \dfrac{1}{k}$. The unstable case is analogous.   
%$$ L^{u}(\phi^t(\theta))=\sqrt{1+\alpha_{u}^{2}(\phi^t(\theta))}=\sqrt{1+4k^2},$$
%for all $t\geq \dfrac{\log 3}{2k}$.
\end{proof}

%%%%%%%%%%%%%%%%%%%%%%%%%%%%%%%%%%%%%%%%%%%%%%%%%%%%%%%%%%%%%%%%%%%%%%%%%
It is important to observe that in {the} Anosov case, subbundles $E^{s}$ and $E^u$ are invariant and Lagrangian, therefore Lemma \ref{Discrete Intersection} can be {written} in terms of the sets $\B^{s,u}$. 
\begin{lemma}[Twist Property]\label{Discrete Intersection 1}
For each $\theta \in SM$, the  sets $\{t\in\mathbb{R}: \phi^{t}(\theta)\in \B^{s}\}$ and  $\{t\in\mathbb{R}: \phi^{t}(\theta)\in \B^{u}\}$ are discrete. 
\end{lemma}

When the manifold has  curvature bounded below, thanks to Lemma \ref{Mane Lemma} and Lemma \ref{Knieper Lemma},  the sets $\B^{s,u}$ have the following special  property:
\begin{lemma}[Transfer Property]\label{Transfer Property}
If $\theta \in SM$ is such that there is $t_0$  with $\phi^{t_0}(\theta)\in \B^{s}$, then there {exists} $t_1\in [t_0 - 1 , t_0 + \sigma]$ such that $\phi^{t_1}(\theta)\in \B^{u}$, where $\sigma$ is as in Lemma \emph{\ref{Knieper Lemma}}. An analogous result for the unstable case.
\end{lemma} 
\begin{proof}
{ From Lemma \ref{Knieper Lemma} it follows that the geodesic arc $\gamma_{\theta}:[t_0 - 1 , t_0 + \sigma] \to M$
has conjugate points. Now suppose that $\phi^{t}(\theta) \notin \B^{u}$ for all $t\in [t_0 - 1 , t_0 + \sigma]$. Then, from Lemma \ref{Mane Lemma} follows that the geodesic arc $\gamma_{\theta}:[t_0 - 1, t_0 + \sigma] \to M$ does not contain conjugate points. This contradiction concludes the proof.}
\end{proof}

\subsection{Topological properties of the first positive (negative) time to $\B^{s,u}$}\label{First time}
%The lemmas presented in this subsection show the most important properties of the sets $\B^{s,u}$, which estimate the first the moment that orbits intersect the sets  $\B^{s,u}$, and will be fundamental tools for the proof of Theorem \ref{main2}. 
In this section, we estimate the first moment that orbits intersect the sets  $\B^{s,u}$, which will be a fundamental tool to prove Theorem \ref{main1}.

We consider the special sets 

$$\B^{u}_{+}=\Big\lbrace\theta\in SM: \text{there is} \,\, t\geq 0 \,\, \text{with}\, \, \phi^{t}(\theta)\in \B^{u}\Big \rbrace.$$
and 
$$\B^{s}_{-}=\Big\lbrace\theta\in SM: \text{there is} \,\, t\leq 0 \,\, \text{with}\, \, \phi^{t}(\theta)\in \B^{s}\Big \rbrace.$$

\noindent For $\theta\in \B^{u}_{+}$, denote by $$t_{+}^{u}(\theta)=\ds\min\, \Big\lbrace t\geq 0: \phi^t(\theta)\in \B^{u}\Big\rbrace, \,\,\text{and} \,\, t_{-}^{u}(\theta)=\ds\inf\, \Big\lbrace t<0: \phi^t(\theta)\in \B^{u}\Big\rbrace.$$ 
Note that if $\phi^{-t}(\theta)\notin \B^{u}$ for all $t<0$, then $t_{-}^{u}(\theta)=-\infty$.
\noindent Analogously, if  $\theta\in \B^{s}_{-}$, denote by $$t_{-}^{s}(\theta)=\ds\max \, \Big\lbrace t\leq 0: \phi^t(\theta)\in \B^{s}\Big\rbrace \,\,\text{and}\,\, t_{+}^{s}(\theta)=\ds\inf \, \Big\lbrace t>0: \phi^t(\theta)\in \B^{s}\Big\rbrace.$$
Note that if $\phi^{t}(\theta)\notin \B^{s}$ for all $t>0$, then $t_{+}^{s}(\theta)=+\infty$. 

%For each $\theta\in \B^{u}_{+}$ we denote by there is $\xi^u\in E^u(\theta)$ with $\|\xi^u\|=1$ such that $J^u_{\xi^u}(t_{+}^{u})=0$. 
From Lemma \ref{Discrete Intersection 1} the times $t_{+}^{u}(\theta)$ and $t_{-}^{s}(\theta)$ are well defined. 

\begin{R}\label{RN1}
From the definition of $t_{-}^u(\theta)$ and $t_{+}^{u}(\theta)$, there is no unstable Jacobi field with zeros in $(t_{-}^u(\theta), t_{+}^u(\theta))$. Analogous result for the stable case.
\end{R}
\begin{Defi}

A subset $A\subset SM$ is called a Pointwise Negatively \emph{(}resp. Positively\emph{)}  Capturing Set if, for every $x\in A$ such that $\mathcal{O}^{-}(x)\subset A$ \emph{(}resp. $\mathcal{O}^{+}(x)\subset A$\emph{)} , it follows that $\mathcal{O}(x)\subset A$, where $\mathcal{O}^{-}(x)$ and $\mathcal{O}^{+}(x)$ denote the negative and positive orbits of $x$, respectively, and $\mathcal{O}(x)$ denotes the orbit of $x$.
\end{Defi}
It is worth noting that negatively or positively invariant sets are not always invariant and therefore cannot be considered Pointwise Negatively \emph{(}or Positively\emph{)} Capturing Sets. This is a rather special property found only in very specific cases of dynamical systems. The following theorem demonstrates that the complements of sets $\B^{u}$ and $\B^{s}$ possess this additional property. Moreover, it is pivotal as it allows us to deduce the absence of conjugate points along a geodesic if they are confined to only one side. Its proof ingeniously utilizes the special times $t_{+}^{u}(\theta)$ and $t_{-}^{s}(\theta)$.
\begin{T}\label{PNPCS}
The set $SM\setminus \B^{u}$ is Pointwise Negatively  Capturing Set and the set $SM\setminus\B^s$ is Pointwise  Positively Capturing Set.
\end{T}
The proof of the previous theorem is somewhat delicate and requires several additional lemmas. Moreover, it will be instrumental in demonstrating straightforwardly that $SM\setminus \B^{u}$ and $SM\setminus \B^{s}$ are non-empty for non-compact manifolds.\\
\ \\
For the rest of the paper, consider the \emph{Wroskian}  $W(J,Y)(t)$ of two Jacobi fields, $J$ and $Y$  defined by
$$W(J,Y)(t):=\langle J(t),Y'(t)\rangle - \langle J'(t),Y(t)\rangle.$$
It is not difficult to prove that $W(J,Y)(t)$ is a constant function, which is zero if and only if $J$ and $Y$ are linearly dependent.
\ \\
Additionally, the special times defined earlier pinpoint the initial interval where the geodesic remains free of conjugate points, as illustrated in the following lemma.
%Furthermore, such times determine the first interval where the geodesic has no conjugate points, as we can see in the next lemma.

%The soul of this section are the following two lemmas, which allow local and uniform estimation of times  $t^{u}_{+}(\theta)$ and $t^{s}_{-}(\theta)$.
\begin{lemma}\label{First-Time}
If $\theta\in \B^{u}_{+}\setminus \B^u$ \emph{(}resp. $\theta\in \B^s_{-}\setminus \B^s$\emph{)} %and $t_{+}^{u}(\theta)>0$ \emph{(}resp. $t_{-}^{s}(\theta)<0$
%\emph{)}, 
then $\gamma_{\theta}(t)$ has no conjugate on $(t_{-}^{u}(\theta),t_{+}^{u}(\theta)]$, \emph{(}resp. $[t_{-}^{s}(\theta), t_{+}^{s}(\theta))$\emph{)}.
\end{lemma}
\begin{proof} In the proof does not matter if $t_{-}^{u}(\theta)$ is finite or not. 
By contradiction, assume that there are $t_{-}^{u}(\theta)< a<b\leq t_{+}^{u}(\theta)$ and a nonzero Jacobi field $J$ along $\gamma_{\theta}$ such that $J(a)=J(b)=0$, then from Corollary \ref{C - Transfer Prop} there is a nonzero unstable Jacobi field $J^u$ along $\gamma_\theta$ and $c\in [a,b]$ such that $J^{u}(c)=0$. Then, since $\theta\in \B^{u}_{+}\setminus \B^u$ and the definition of $t_{-}^{u}(\theta)$ we have $c\in (0,t_{+}^{u}(\theta)]$. 
If $c<t_{+}^{u}(\theta)$ we have a contradiction with the minimality of $t_{+}^{u}(\theta)$. Otherwise, $b\geq c=t_{+}^{u}(\theta)\geq b$ and the Wroskian, $W(J, J^u)=0$. Thus it should be $J^u(a)=0$ with $t_{-}^{u}(\theta)<a<b=t_{+}^u(\theta)$, which provides a contradiction from Remark \ref{RN1}.% As $\theta\in \B^{u}_{+}\setminus \B^u$, then  and again we have a contradiction from Remark \ref{RN1}.\\
% Analogously, in the case (ii), by definition $t_{-}^{u}(\theta)$, it should be $c=t_{-}^{u}(\theta)$, but then $t_{-}^{u}(\theta)= a\leq c$ and $W(J, J^u)=0$. Thus should be $J^u(t_{-}^{u}(\theta))=0$ of The stable case is analogous.
\end{proof}
\begin{C}\label{CN1'}
If $\theta\in \B^{u}_{+}\setminus \B^u$ \emph{(}resp. $\theta\in \B^s_{-}\setminus \B^s$\emph{)}, then $\gamma_{\theta}(t)$ has no conjugate on $[0,t_{+}^{u}(\theta)]$, \emph{(}resp. $[t_{-}^{s}(\theta), 0]$\emph{)}.
\end{C}
Before proving Theorem \ref{PNPCS}, we recall an important property regarding closed intervals without conjugate points (cf. \cite[Corollary 2.12]{Knieper}).

%Next, we prove that if conjugate points only appear on the left or right side of a geodesic, it means that there are no conjugate points on both the left and right sides of the geodesic.  For this sake, we shall use the following lemma (see also\cite[Corollary 2.12]{Knieper}).
%%%%%%%%%%%%%%%%%%%%%%%%%%%%%%%%%%%%%%%%%%%%%%%%%%%%%%%%%%%%%%%%%%%
%%%%%%%%%%%%%%%%%%%%%%%%%%%%%%%%%%%%%%%%%%%%%%%%%%%%%%%%%%%%%%%%%%%%%%%%%%
\begin{lemma}\label{Green-Lemma}
	Let $M$ be a Riemannian manifold with curvature bounded below by $-k^2$ and $T\geq 1$. Then there exists $A=A(k)>0$ such that if $\gamma\colon [-1,T+1]\to M$ is a geodesic arc without conjugate points and $J$, $-1\leq t\leq T+1$ is a perpendicular Jacobi field on $\gamma$ with $J(0)=0$ and %\tcb{$\|J'(0)\|=1$}, 
	then $$\|J'(t)\|\leq A\|J(t)\|$$
for all $1\leq t\leq T$.
\end{lemma}

\begin{proof}[\textbf{\emph{Proof of Theorem \ref{PNPCS}}}]
Let us prove that $SM\setminus \B^u$ is pointwise negatively capturing set since the other case is analogous. Assume that $\mathcal{O}^{-}(\theta)\subset (SM\setminus \B^u)$. Then, considering a fixed number $\beta>1$ and put $\theta_{\beta}:=\phi^{-\beta}(\theta)\notin \B^u$ and  $\mathcal{O}^{-}(\theta_{\beta})\subset (SM\setminus \B^u)$. Moreover, from Lemma \ref{Mane Lemma} the geodesic $\gamma_{\theta_{\beta}}(t)$ has no conjugate points in $(-\infty, \beta)$. Therefore, without loss of generality, changing $\theta_{\beta}$ by $\theta$, we can assume that $\theta\notin \B^u$ and $\gamma_{\theta}(t)$ has no conjugate points in $(-\infty, \beta)$. %If necessary, we iterated negatively $\theta$ and used the Lemma \ref{Discrete Intersection} to guarantee $\beta>1$ and $\theta\notin \B^u$. 
% \textcolor{red}{In this case, if $\theta \notin  \mathcal{B}^{u}_{+}$, then $\mathcal{O}(\theta)\subse SM\setminus \mathcal{B}^u$. from Lemma \ref{Mane Lemma} we have nothing to do.}  
 By contradiction assume that $\mathcal{O}^{+}(\theta)\cap \B^u \neq \emptyset$, or equivalent $\theta \in \mathcal{B}_{u}^{+}\setminus \B^u$, consequently  $t_{+}^{u}(\theta)<\infty$.\\
From Lemma \ref{First-Time} $\gamma_{\theta}$ has no conjugate points on $[0,t_{+}^{u}(\theta)]$.\\
\ \\
\textbf{Claim:} $\gamma_{\theta}$  has no conjugate points on $(-\infty,t_{+}^{u}(\theta)]$. 
\begin{proof}[\emph{\textbf{Proof of Claim}}]
%We can assume, without loss of generality, that $\beta>1$, since we can change $\theta$ by $\phi^{-r}(\theta)$, for some $r>1$.
If $t_{+}^{u}(\theta)\leq \beta$, we are done. If $t_{+}^{u}(\theta)>\beta$, then by contradiction, assume that $\gamma_{\theta}(t)$ has conjugate points in  $(-\infty,t_{+}^{u}(\theta)]$, then there is a nonzero Jacobi field $J$ and $a, b \in (-\infty,t_{+}^{u}(\theta)]$ such that $J(a)=J(b)=0$, $a<b$. It is easy to see that $a<0$ and $\beta\leq b$. Thus, from Corollary \ref{C - Transfer Prop} and the minimality of $t_{+}^{u}(\theta)$ we have $\phi^{r_0}(\theta)\in \B^{u}$ for some $r_{0}\in (a,0)$. So, by Lemma \ref{Knieper Lemma} the geodesic $\gamma_{\phi^{r_0}(\theta)}(t)=\gamma_{\theta}(t+r_0)$ has  conjugate points in $[-\sigma, 1]$. Therefore, $\gamma_{\theta}(t)$ has conjugate points in $[-\sigma+r_0, 1+r_0]$ which implies, since $r_0<0$ and $\beta>1$, that $\gamma_{\theta}(t)$ has conjugate points in $[-\sigma+r_0, \beta)$ which provides a contradiction and the proof of claim is concluded. \end{proof}
\indent Now we denote by $J^{u}_{\theta}(t)$ the unstable Jacobi field along $\gamma_{\theta}(t)$ such that $J^u_{\theta}(t_{+}^{u}(\theta))=0$. For $r>0$, $J^u_{r}(t):=J^u_{\theta}(t-r)$ is an unstable Jacobi field along 
$\gamma_{\phi^{-r}(\theta)}(t)$ with $t_{+}^{u}(\phi^{-r}(\theta))=t_{+}^{u}(\theta)+r$. We denote $\xi^u_{r}=(J^u_{r}(0), (J^u_{r})'(0))$ and assume that $\|\xi^{u}_r\|=1$. Let $J^s_{r}(t)$ be a stable Jacobi field along $\gamma_{\phi^{-r}(\theta)}(t)$ with $J^s_{r}(0)=J^u_{r}(0)$ and put $\xi^s_{r}=(J^s_{r}(0), (J^s_{r})'(0))$.\\
We define the Jacobi field $J_r(t)=J^{u}_{r}(t)-J^{s}_{r}(t)$ which satisfies $J_r(0)=0$. Take $r>0$ such that $t_{+}^{u}(\theta)+r-1>1$, since $\gamma_{\phi^{-r}(\theta)}(t)$ has no conjugate points in $(-\infty, t_{+}^{u}(\theta)+r)$, from Lemma \ref{Green-Lemma}  $$\|J'_{r}(t)\|\leq A\|J_{r}(t)\| \,\, \, \, \, \, \text{for} \,\,\, 1\leq t < t_{+}^{u}(\theta)+r-1. $$
%since $\gamma_{\phi^{-r}(\theta)}(t)$ has no conjugate points in $(-\infty, t_{+}^{u}(\theta)+r)$.\\
Thus, 
\begin{eqnarray}\label{eq1-Bounded t_{+1}}
\|(J^{u}_{r})'(t)\|-\|(J^{s}_{r})'(t)\|\leq \|J'_{r}(t)\|\leq A\|J_r(t)\| \leq A\|J^{u}_{r}(t)\|+ A\|J^{s}_{r}(t)\|
\end{eqnarray}
for all $1\leq t < t_{+}^{u}(\theta)+r-1$.\\
Therefore, by the definition of Anosov geodesic flow, we have 
\begin{eqnarray}\label{EQN1-1}
\max\{\|J^{s}_{r}(t)\|\,, \|(J^{s}_{r})'(t)\|\}&\leq & \Big( \|J^{s}_{r}(t)\|^{2}+\|(J^{s}_{r})'(t)\|^{2}\Big)^{\frac{1}{2}} \nonumber \\ &=&\Big\|D\phi_{\theta}^{t}(\xi^s_{r})\Big\|\leq C\lambda^{t}\|\xi^{s}_{r}\|.
\end{eqnarray}
Thus, from  (\ref{eq1-Bounded t_{+1}}) and (\ref{EQN1-1}) we have 
\begin{equation*}\label{eq2-Bounded t_{+2}}
\|(J^{u}_r)'(t)\|  \leq A\|J^{u}_{r}(t)\|+ C(A+1)\lambda^{t}\|\xi^s_{r}\| 
\end{equation*}
for all $1\leq t < t_{+}^{u}(\theta)+r-1$.\\
Therefore,  as $\|\xi^{u}_{r}\|=1$ 
\begin{eqnarray}\label{eq3-Bounded}
\frac{1}{C}\lambda^{-t}&\leq& \|D\phi^{t}(\xi^u_{r})\| \nonumber\\
 &\leq& \Big( \|J^{u}_{r}(t)\|^{2}+\|(J^{u}_{r})'(t)\|^{2} \Big)^{\frac{1}{2}}\nonumber \\ 
&\leq & \Big(\|J^{u}_{r}(t)\|^{2}+ \Big( A\|J^{u}_{r}(t)\|+ C(A+1)\lambda^{t}\|\xi^s\| \Big)^{2}\Big)^{\frac{1}{2}} \nonumber \\
&\leq &\ \sqrt{1+A^2}\|J^{u}_{r}(t)\|+ \sqrt{2CA(A+1)\lambda^{t}\|J^{u}_{r}(t)\|\|\xi^s_{r}\|} + C(A+1)\lambda^{t}\|\xi^s_{r}\| \nonumber \\
&\leq & \sqrt{1+A^2}\|J^{u}_{r}(t)\|+ \sqrt{2CA(A+1)\|J^{u}_{r}(t)\|\|\xi^s_{r}\|} + C(A+1)\lambda^{t}\|\xi^s_{r}\|,
\end{eqnarray}
whenever $1\leq t < t_{+}^{u}(\theta)+r-1$. \\
Note that  $\phi^{-r}(\theta)\notin \B^s$ for all $r>0$, then from Corollary \ref{Unif Bounded t_{+}*} we have $$\|\xi^s_{r}\|\leq \sqrt{1+4k^2}\|J^s_{r}(0)\|\leq \sqrt{1+4k^2}\|\xi^u_{r}\|= \sqrt{1+4k^2}, \,\,\, \text{for}\,\,\ r>\frac{1}{k}.$$
Letting $t\to t^{u}_{+}(\theta)+r-1$, then $J_{r}^{u}(t)\to J_{\theta}^{u}(t^{u}_{+}(\theta)-1)$ and from (\ref{eq3-Bounded}) 
\begin{eqnarray*}
\frac{1}{C}\lambda^{-t^{u}_{+}(\theta)-r+1}&\leq&  \sqrt{1+A^2}\|J^{u}_{\theta}(t^{u}_{+}(\theta)-1)\|\\ &+& \sqrt{2CA(A+1)\|J^{u}_{\theta}(t^{u}_{+}(\theta)-1)\|\sqrt{1+4k^2}} \\ &+& C(A+1)\lambda^{t_{+}^{u}(\theta)+r-1}\sqrt{1+4k^2}, \,\,\, \text{for}\,\,\ r>\frac{1}{k}.
\end{eqnarray*}
Taking $r$ large enough, the two last inequalities provide a contradiction, since $0<\lambda<1$. Therefore, $\mathcal{O}(\theta)\subset SM\setminus \B^u$ as we wished.
\end{proof}

\begin{C}\label{NewLemma1*}
Let $\gamma_{\theta}(t)$ be a geodesic without conjugate points in $(-\infty, \beta)$ or $(\beta, +\infty)$, for some $\beta$, then $\gamma_{\theta}(t)$ has no conjugate points in $(-\infty, +\infty)$. 
\end{C}
\begin{proof}
%Let us prove the case $(-\infty, \beta)$, since the other case is analogous. 
Assume that $\gamma_{\theta}(t)$ has no conjugate points in $(-\infty, \beta)$, the other case is analogous. From Lemma \ref{Mane Lemma}, $\mathcal{O}^{-}(\theta)\subset SM\setminus \B^u$. So, Theorem \ref{PNPCS} provides $\mathcal{O}(\theta)\subset SM\setminus \B^u$ and consequently $\gamma_{\theta}(t)$ has no conjugate points.
\end{proof}
%\tcr{nao sei see matter estes lemas}\\
%\tcr{Using the notation of Lemma \ref{Knieper Lemma} we can write Theorem \ref{PNPCS} as follows:
%\begin{C}\label{New Lemma Nomenclature}
%Let $\theta\in SM$ be  such that there is $\beta$ such that 
%$$E^s(\phi^{t}(\theta))\cap V(\phi^{t}(\theta))=\{0\}, \, \, \emph{for all}\, \,t<\beta \,\,\emph{or}\,\, t>\beta,$$ then 
%$$E^s(\phi^{t}(\theta))\cap V(\phi^{t}(\theta))=\{0\}, \,\, \emph{for all} \,\, t\in \re,$$
%and consequently, $\gamma$ has no conjugate points. Analogous result for the unstable case.
%\end{C}
%\tcr{PAREI AQUI}\\
%}
\noindent The last corollary has the following important consequence. 
\begin{lemma}\label{Nonempty-Good-Set} If $M$ is a non-compact manifold with Anosos geodesic flow, then there is $\theta \in SM$ such that the geodesic $\gamma_{\theta}(t)$ has no conjugate points. Consequently, 
$$E^{s,u}(\phi^{t}(\theta))\cap V(\phi^{t}(\theta))=\{0\}, \,\, \emph{for all} \,\, t\in \re.$$
\end{lemma}
\begin{proof}
 Since  $M$ is a non-compact manifold,  there {exists} a \emph{ray} $\gamma_{\theta}: [0,\infty)\to M$, \emph{i.e.}, $\gamma_{\theta}$ is a geodesic such that $d(\gamma_{\theta}(t), \gamma_{\theta}(s))=|t-s|$, which implies that $\gamma_{\theta}$ does not have conjugate points in $(0, +\infty)$, then from Corollary \ref{NewLemma1*} the geodesic $\gamma_{\theta}$ does not have conjugate points in $(-\infty,+\infty)$.
\end{proof}
This lemma shows that the set of points $\theta$ such that its orbit never intersects $\B^s$ and $\B^u$ is nonempty. It will be used in Section \ref{Main-Results} to prove Theorem \ref{T-MAIN1}.

\subsubsection{Closedness of $\B^{u}_{+}$ and $\B^s_{-}$}
The main goal of this section is to establish that $\mathcal{B}^{u}_{+}$ and $\mathcal{B}^{s}_{-}$ are closed subsets of $SM$ (see Lemma \ref{Bounded t_{+}} and Lemma \ref{Bounded t_{+}-dim 2}). This assertion is pivotal and intricate, requiring the construction of several auxiliary results to substantiate it.
%%%%%%%%%%%%%%%%%%%%%%%%%%%%%%%%%%%%%%%%%%%%%%%%%%%%%%%%%%%%%%%%%%%%%%%%%%
%%%%%%%%%%%%%%%%%%%%%%%%%%%%%%%%%%%%%%%%%%%%%%%%%%%%%%%%%%%%%%%%%%%%%%%%%

%To find other properties of the set $\B^{u}_{+}$ and $\B^{s}_{-}$, we
Consider the diffeomorphism  $\mathcal{I}\colon SM \to SM$, given by 
$$\mathcal{I}(x,v)=(x,-v).$$
%\textcolor{red}{If $\theta=(x,v)$, we denote $\theta^{*}=(x,-v)$, then $\mathcal{I}(\theta)=\theta^{*}$.} 
Note that $\mathcal{I}^{2}=\text{Id}$. This diffeomorphism helps to relate $\B^{u}_{+}$ and $\B^{s}_{-}$, and the non-positive time $t_{+}^{u}(\theta)$ and non-negative time $t_{-}^{s}(\mathcal{I}(\theta))$. 

\begin{lemma}\label{I(u)=s}
The diffeomorphism $\mathcal{I}$ has the following properties
\begin{itemize}
\item[\emph{(i)}] $D\mathcal{I}_{\theta}(E^{u}(\theta))=E^{s}(\mathcal{I}(\theta))$ and $D\mathcal{I}_{\theta}(E^{s}(\theta))=E^{u}(\mathcal{I}(\theta))$.
\item[\emph{(ii)}] $\mathcal{I}(\B^{s,u})=\B^{u,s}$, respectively. Also, $\mathcal{I}(\B^{u}_{+}) =\B^{s}_{-}$ and $\mathcal{I}(\B^{s}_{-}) =\B^{u}_{+}$.
\item[\emph{(iii)}] $t_{+}^{u}(z)=-t_{-}^{s}(\mathcal{I}(z))$ and $t_{-}^{s}(z)=-t_{+}^{u}(\mathcal{I}(z))$.
\end{itemize}
%The sets $\B^{u}_{+}$ and $\B^{s}_{-}$ satisfies $$\mathcal{I}(\B^{u}_{+}) =\B^{s}_{-} \, \, \text{and} \, \, \, \mathcal{I}(\B^{s}_{-}) =\B^{u}_{+}$$.%The image of  $\B^{u}_{+}$ by $\mathcal{I}$ is $\B^{s}_{-}$.
\end{lemma}
\begin{proof}
It is easy to see that in the horizontal and vertical coordinates $D\mathcal{I}_{\theta}(\xi_1,\xi_2)=(\xi_1,-\xi_2)$. So, if $\xi\in E^u(\theta)$ %, then $D\mathcal{I}_{\theta}(\xi)\in E^{s}(\mathcal{I}(\theta))$. In fact: 
 and consider $J_{\xi}(t)$ the unstable Jacobi field associated to $\xi$, then $J_{D\mathcal{I}_{\theta}(\xi)}(t)=J_{\xi}(-t)$ is a stable Jacobi field along $\gamma_{_{\mathcal{I}(\theta)}}(t)$, which implies $D\mathcal{I}_{\theta}(\xi)\in E^s(\mathcal{I}(\theta))$. As $\text{dim}\,E^u(\theta)=\text{dim}\,E^{s}(\mathcal{I}(\theta))$, then $D\mathcal{I}_{\theta}(E^{u}(\theta))=E^{s}(\mathcal{I}(\theta))$.  \\
If $\theta \in \B_{+}^{u}$, then there is $\xi \in E^{u}(\theta)$ and $t_{0}\geq 0$ such that $J_{\xi}(t_0)=0$. Therefore, since $J_{D\mathcal{I}_{\theta}(\xi)}(t)$ is a stable Jacobi field along $\gamma_{_{\mathcal{I}(\theta)}}(t)$ and $J_{D\mathcal{I}_{\theta}(\xi)}(-t_{0})=J_{\xi}(t_{0})=0$, then $\mathcal{I}(\theta)\in \B_{-}^{s}$.
Since $\mathcal{I}^{2}=\text{Id}$, the other cases follow immediately.
\end{proof}

In the following lemma, we establish a specific lower bound function for the norm of the Jacobi field with a zero in an interval without conjugate points. The proof of this lemma follows the same lines as \cite[Lemma 2.13]{Knieper}, with a suitable modification to obtain more accurate estimates needed for the proof of Lemma \ref{Bounded t_{+}} and Lemma \ref{Bounded t_{+}-dim 2}.
\begin{lemma}\label{Green-Lemma}
Let $M$ be a Riemannian manifold with curvature bounded below by $-k^2$, $T >1$, and $\gamma: [-1, T] \to M$ a geodesic arc without conjugate points. Then there exists a positive real function $\rho: (0, T-1) \to M$, depending only on $k$, such that for all  perpendicular Jacobi field $J$ with $J(0) = 0$
$$||J(r)||^2 \geq \rho(T-r) , \, \, {\rm{for}} \, \, {\rm{all}} \, \, 1 \leq r <  T,$$
where% $J$ is a perpendicular Jacobi field with $J(0) = 0$ and
$$ \rho(t)= \displaystyle\frac{||J'(0)||^4}{\Big( \displaystyle\frac{8}{3}||J'(0)||^2 + \displaystyle\frac{16}{15}k^2\Big) \Big(  k \coth k + \displaystyle\frac{1}{t} + \displaystyle\frac{k^2t}{3}        \Big)  }.$$

\end{lemma}
\begin{proof}
	Fix $r \in [1, T)$ and $\delta \in (0, T-r)$. Now consider the piecewise differentiable vector field $X$ along $\gamma(t)$ defined by
\[
X(t)=
\begin{cases}
0,&\mbox{ }\quad -1 \leq t \leq 0,\\
J(t), &\mbox{ }\quad 0 < t \leq r,\\
\Big(1 - \frac{(t-r)}{\delta}\Big)V(t)  &\mbox{ }\quad r < t \leq r + \delta,
\end{cases}
\]

%	$$ X(t) = \left\{
%	\begin{array}{rcl}
%		0, \,\,\,\,\ \,\,\,\,\,\,\,\, \hspace{2cm}& & -1 \leq t \leq 0, \\
%		J(t),  \,\,\,\,\ \,\,\,\,\,\,\,\, \hspace{2cm} &  & 0 < t \leq r, \\
%		(1 - \frac{(t-r)}{\delta})J(t) & & r < t \leq r + \delta,	 
%	\end{array}
%	\right.
%$$
\noindent where $V(t)$ is the parallel vector field along $\gamma$ in $[r, r+\delta)$ such that $V(r) = J(r)$ .
\noindent Then, from the index form 
$$ 0 < I_{[-1, r+ \delta]}(X,X) = \langle J'(r), J(r) \rangle + \dfrac{1}{\delta} \langle J(r), J(r) \rangle - \displaystyle\int_{r}^{r+ \delta} \langle R(t)X(t), X(t) \rangle \, dt.$$
From the proof of Corollary 2.12 in \cite{Knieper}, we have $\langle J'(r), J(r)\rangle \leq k \coth k ||J(r)||^2$. Hence, 
\begin{eqnarray}\label{eq1NewKnieper}
	 I_{[-1, r+ \delta]}(X,X) & \leq & ||J(r)||^2\Big(k \coth k + \dfrac{1}{\delta}\Big) + k^2\displaystyle\int_{r}^{r+ \delta} \langle X(t), X(t) \rangle \, dt \nonumber\\
	 & = & ||J(r)||^2 \Big(k \coth k + \dfrac{1}{\delta} + k^2\displaystyle\int_{r}^{r+ \delta}  	\Big(1 - \frac{(t-r)}{\delta}\Big)^2 \, dt \Big) \nonumber\\
	 & =& ||J(r)||^2 \Big(k \coth k + \dfrac{1}{\delta} + \dfrac{k^2 \delta}{3} \Big).
\end{eqnarray}

\noindent Let $Y$ be the piecewise differentiable vector field along $\gamma(t)$ defined by

\[
Y(t)=
\begin{cases}
(1-t^2)Z(t),&\mbox{ }\quad -1 \leq t \leq 1,\\
0, &\mbox{ }\quad 0 < t \leq r + \delta.
\end{cases}
\]

\noindent 	where $Z$ is the parallel vector field along $\gamma$ in $[-1,1]$ such that $Z(0) = J'(0)$. 
\noindent  Then
$$    I_{[-1, r+ \delta]}(X,Y) =  - ||J'(0)||^2.                                                               $$
Furthermore, it is not difficult to prove that 
\begin{equation}\label{eq2NewKnieper}
0 < I_{[-1, r+ \delta]}(Y,Y) \leq \displaystyle\frac{8}{3}||J'(0)||^2 + \displaystyle\frac{16}{15}k^2.
\end{equation}

\noindent Now consider for $\lambda \in \mathbb{R}$ the quadratic equation
$$       I_{[-1, r+ \delta]}(X + \lambda Y, X + \lambda Y)  =     I_{[-1, r+ \delta]}(X, X)            + 2\lambda    I_{[-1, r+ \delta]}(Y,X)  + \lambda^2    I_{[-1, r+ \delta]}(Y,Y).$$
This quadratic equation has no real roots since the Jacobi equation has no conjugate points,  therefore 
$$         I_{[-1, r+ \delta]}(X,X)   I_{[-1, r+ \delta]}(Y,Y)      >    I_{[-1, r+ \delta]}(Y,X) =  ||J'(0)||^4.                                              $$
From (\ref{eq1NewKnieper}) and (\ref{eq2NewKnieper}) we have 
$$ ||J(r)||^2 \geq  \displaystyle\frac{||J'(0)||^4}{\Big(\displaystyle\frac{8}{3}||J'(0)||^2 + \displaystyle\frac{16}{15}k^2\Big) \Big(  k \coth k + \displaystyle\frac{1}{\delta} + \displaystyle\frac{k^2 \delta}{3}\Big) }.$$
Hence, taking $\delta \to T-r$ we conclude the proof of lemma.
\end{proof}
% \begin{lemma}\label{Green-Lemma}
%	Let $M$ be a Riemannian manifold with curvature bounded below by $-k^2$ and $T\geq 1$. Then there exists $\rho>0$ depending only on $k$ such that if $\gamma\colon [-1,T+1]\to M$ is a geodesic arc without conjugate points and $J$ is a perpendicular Jacobi field on $\gamma$ with $J(0)=0$ and %$\|J'(0)\|=1$, then $$\|J'(t)\|\leq A\|J(t)\|$$
%	$$||J(t)||\geq \rho||J'(0)||^2, \,\,\,\ \text{for all}\,\,\, 1\leq t\leq T.$$
%\end{lemma}

%\textcolor{red}{PAREI AQUI 17/04/2023 17:57}

The lemma below demonstrates that $\mathcal{B}^{u}_{+}$ and $\mathcal{B}^{s}_{-}$ are closed sets in higher dimensions under the assumption of bounded sectional curvature. Interestingly, this result holds in dimension two under the sole assumption of sectional curvature being bounded below (see Lemma \ref{Bounded t_{+}-dim 2}).
\begin{lemma}\label{Bounded t_{+}}
If the sectional curvature of $M$ is bounded, then the sets $\mathcal{B}^{u}_{+}$ and $\mathcal{B}^{s}_{-}$ are closed.
\end{lemma}

\begin{proof}[\emph{\textbf{Proof of Lemma \ref{Bounded t_{+}}}}]
As $\mathcal{I}$ is a diffeomorphism, from Lemma \ref{I(u)=s} it is sufficient to prove that $\B^{u}_{+}$ is closed.\\ 
Assume that $\theta_n \to \theta$ with $\theta_n\in \mathcal{B}^{u}_{+}$, then for each $n$ there is $t_n\geq 0$ such that $\phi^{t_n}(\theta_n)\in \B^{u}$ and there is $\eta_{n}\in E^u(\theta_n)$, $||d\phi^{1}(\eta_n)||=1$ such that $J_{\eta_n}(t_n)=0$. %The parameter $\beta$ will be chosen accordingly later.  %\textcolor{blue}{Without loss of generality, we can assume that $t_{n}=t^{u}_{+}(\theta_n)\geq 4$.} 
Moreover, from Lemma \ref{stable-norm} and Lemma \ref{Discrete Intersection 1} we can assume that $\theta \notin \B^u\cup \B^s$ and for $n$ large enough $\theta_n\in \mathcal{U}^u_{\theta}\cap \mathcal{U}^s_{\theta} \subset SM\setminus (\B^u\cup\B^s)$.\\
\indent It is clear that if $\phi^1(\theta)\in \B^u$ then $\theta\in \B_{+}^{u}$.  If $\phi^1(\theta)\in \B^s$, then by the \emph{Transfer Property} (Lemma \ref{Transfer Property}) there is $t_1\in [0,1+\sigma]$ such that $\phi^{t_1}(\theta)\in \B^u$ and $\theta\in \B_{+}^{u}$. Therefore, we can assume that $\phi^{1}(\theta)\notin (\B^s\cup \B^u)$, and consequently $\phi^1(\theta_n)\in \mathcal{U}^u_{\phi^1(\theta)}\cap \mathcal{U}^s_{\phi^1(\theta)} \subset SM\setminus (\B^u\cup\B^s)$. From Lemma \ref{First-Time} provides that $\gamma_{\theta_n}(t)$ has no conjugate points on $[0,t_n]$. Consider for each $n$, $\theta_{n}^1:=\phi^1(\theta_{n})$, so $\gamma_{\theta_n^1}(t)$ has no conjugate points on $[-1,t_n-1]$. \\

\noindent \textbf{Main Claim:} The sequence $\{t_n\}$ is bounded.\\ 
\ \\
\noindent Note that this claim implies that, from Lemma \ref{Le-Closed} and passing a sub-sequence if necessary, $\phi^{t_{n}}(\theta_{n}) \to \phi^{t_0}(\theta)\in \B^u$, for some $t_0\geq 0$. Thus we have that $\theta \in \B_{+}^{u}$ as we wish.\\
\ \\
\indent Therefore, to complete the proof of the Lemma, we need only prove the Main Claim. 

% The proof of this claim is a little delicate and so long, so we split its proof into 10 sub-claim need to complete such proof.
\begin{proof}[\emph{\textbf{Proof of Main Claim}}]
By contradiction, suppose that $\{t_n\}$ is unbounded. Since \\
$\eta_n=d\phi^{-1}_{\phi^{1}(\theta)}(d\phi^{1}_{\theta}(\eta_n))$, then $|\eta_n|\leq C\lambda ||d\phi^{1}_{\theta}(\eta_n)||=C\lambda$, then without loss of generality, passing to a sub-sequence if necessary,  we assume that $\eta_n\to \eta \in E^u(\theta)$, and $\theta_n^1\to \theta^1:=\phi^{1}(\theta)$.
%\tcr{Taking $\epsilon>0$ small we have that $\gamma_{\theta_{n}^1}$  has no conjugate points on $[-1,t_n-1-\epsilon]$.} 
Let $\xi_n\in E^s(\theta_n)$ such that $J_{\xi_{n}}(1)=J_{\eta_n}(1)$. Note that $E^s(\theta_n^1)\cap E^u(\theta_n^1)=\{0\}$, then $J'_{\xi_{n}}(1)\neq J'_{\eta_n}(1)$. Thus, we define 
 \begin{equation*}\label{eq(-1)New}
 J_{n}(t)=J_{\eta_n}(t+1)-J_{\xi_n}(t+1)
 \end{equation*} 
 % J_{n}(t)=a\frac{J_{\eta_n}(t+1)-J_{\xi_n}(t+1)}{{||J'_{\eta_{n}}(1)- J'_{\xi_n}(1)||}}
a non-zero Jacobi field along $\gamma_{\theta_n^1}(t)$ with $J_n(0)=0$.
Put  $||J'_{\eta_{n}}(1)- J'_{\xi_n}(1)||:=\kappa_n$.  %\tcr{The parameter $\kappa_n$ will be chosen accordingly later}.\\
From (\ref{eq4}) and continuity of $E^u$,  $d\phi^1_{\theta_{n}}(\eta_n)=(J_{\eta_{n}}(1), J'_{\eta_n}(1))\to d\phi_{\theta}^1(\eta)=(J_{\eta}(1), J'_{\eta}(1))\in E^{u}(\theta^{1})$. Thus $J_{\xi_n}(1)=J_{\eta_n}(1)\to J_{\eta}(1)$. Also, from Lemma \ref{stable-norm}, as $\phi^1(\theta_n)\in \mathcal{U}^u_{\phi^1(\theta)}\cap \mathcal{U}^s_{\phi^1(\theta)}$ we have 
$$||J'_{\xi_n}(1)||\leq \alpha_{s}(\theta^1)||J_{\xi_n}(1)|| \,\,\, \text{and}\,\,\, ||J'_{\eta_n}(1)||\leq \alpha_{u}(\theta^1)||J_{\eta_n}(1)||.$$
So, passing a subsequence if necessary we can assume that \newline $\xi_n=d\phi^{-1}_{\phi^1(\theta_{n})}(J_{\xi_n}(1),J'_{\xi_n}(1))\to \xi\in E^s(\theta)$. \\
Remember the \emph{Wroskian}, $W(J_{\xi_n},J_{\eta_n})(t)$ of $J_{\xi_n}$ and $J_{\eta_n} $ defined by
$$W(J_{\xi_n},J_{\eta_n})(t):=\langle J_{\xi_n}(t),J'_{\eta_n}(t)\rangle - \langle J'_{\xi_n}(t),J_{\eta_n}(t)\rangle.$$
Since $J_{\xi_n}$ and $J_{\eta_n}$ are linearly independent, it is evident that $W(J_{\xi_n},J_{\eta_n})(t)$ is a nonzero constant function. Therefore, given $J_{\eta_n}(t_n)=0$, we have:
\begin{eqnarray*}\label{EQ0''N}
0\neq W_n&:=& W(J_{\xi_n},J_{\eta_n})(t)=\langle J_{\xi_n}(t_n),J'_{\eta_n}(t_n)\rangle \nonumber\\
&\,=&\langle J_{\xi_n}(1),J'_{\eta_n}(1)\rangle - \langle J'_{\xi_n}(1),J_{\eta_n}(1)\rangle \nonumber \\
&\,=&\langle J_{\xi_n}(1),J'_{\eta_n}(1) -  J'_{\xi_n}(1)\rangle.
\end{eqnarray*} 
As $\xi_n\to \xi$ and $\eta_n \to \eta$, then $W_n$ converges to $\langle J_{\xi}(1),J'_{\eta}(1) - J'_{\xi}(1)\rangle:= W_0=W(J_{\xi}, J_{\eta})\neq 0$. Thus, for sufficiently large $n$, we have:

\begin{equation}\label{EQ0'NN}
\frac32 |W_0|\geq |W_n|\geq \frac12|W_0|.
\end{equation}
%%%%%%%%%%%%%%%%%%%%%%%%%%%%%%%%%%%%%%%%%%%%%%%%%%%%%%%%%%%%%%%%%%%%%%%%%%%%
In the remainder of the proof, without loss of generality, we assume $W_0>0$ (this case always occurs in dimension two), since the case of $W_0<0$ is analogous.

For each $n$, we will now introduce four sequences of parameters $\epsilon_n^2$, $\beta_n$, $\sigma_n$, and $\mu_n$. These sequences enable us to construct a suitable Taylor polynomial approximation of $||J_n(t)||^2$ in an appropriate neighborhood of $t_n$, which depends on these parameters. This will allow us to derive a contradiction with the assistance of Lemma \ref{Green-Lemma}. \\
%%%%%%%%%%%%%%%%%%%%%%%%%%%%%%%%%%%%%%%%%%%%%%%%%%%%%%%%%%%%%%%%%%%%%%%%%%%
\textbf{Parameters $\epsilon_n^2$ and $\beta_n$:}\\ 
\noindent For each $n$, we consider the function $f_{n}(t):=||J_{n}(t)||^2$.
Then $f_n(t_n-1)=||J_{\xi_{n}}(t_n)||^2:=\epsilon_n^2$ and 
\begin{eqnarray}\label{E1N}
f'_{n}(t_n-1)&=&2\langle J'_{\eta_n}(t_n)-J'_{\xi_n}(t_n), -J_{\xi_{n}}(t_n)\rangle\nonumber \\
&=& -2\langle J'_{\eta_n}(t_n), J_{\xi_{n}}(t_n)\rangle + 2\langle J'_{\xi_n}(t_n), J_{\xi_{n}}(t_n)\rangle \nonumber \\
&=& -2W_n+ 2\langle J'_{\xi_n}(t_n), J_{\xi_{n}}(t_n)\rangle.
\end{eqnarray}
From (\ref{E1N}), since $J_{\xi_n}(t)$ is a stable Jacobi field, then for $n$ large enough $f_n'(t_n-1)<0$.
In this case, consider the parameter
$$\beta_n = \inf_{\alpha>0}\{\alpha\in (t_n, \infty): f_n'(t_n-1+\alpha)=0\}.$$
Since $f_n(t)$ is a a non-negative function, $f_n(t_n-1)$ is small, and $J_{\eta_n}(t)$ is an unstable Jacobi field, then $\beta_n$ is well defined. Moreover, from definition of $\beta_n$ 
\begin{equation}\label{E3N}
f'_n(t_n-1+\alpha)<0 \,\,\,\text{for}\,\,\, t\in [0, \beta_n)\,\,\ \text{and}\,\,\, f'_n(t_n-1+\beta_n)=0,
\end{equation}
consequently, for $w\in (t_n-1,t_n-1+\beta_n)$ we have
\begin{equation}\label{E4N}
0\leq f_n(w)<f_{n}(t_n-1)=||J_{\xi_n}(t_n)||^2:=\epsilon_n^2.
\end{equation}
\textbf{Parameters $\sigma_n$ and $\mu_n$:}\\
Consider the function $g_n(t)=\langle J_{\xi_n}(t), J_{\eta_n}(t)\rangle$, which satisfies $g_n(1)>0$, $g_n(t_n)=0$ and $g_n'(t_n)=W_n>0$, for $n$ large enough. %\tcb{Let us assume that $W_0>0$}\\
So, the following parameter is well-defined  

$$\sigma_n=\sup_{\alpha\in (0,t_n]}\{g'_{n}(t)\leq 0;\,\, \, t\in [t_n-\alpha, t_n]\}.$$
Also, if  $h_{n}(t): =||J_{\eta_n}(t)||^2$, then since $J_{\eta_n}(t)$ is aN unstable Jacobi field and $h_{n}(t_n)=0$ the following parameter is well-defined 
$$\mu_n:=\sup_{\alpha\in (0,t_n]}\{h'_{n}(t)\leq 0;\,\, \, t\in [t_n-\alpha, t_n]\}.$$

We will now examine some properties of the parameters $\epsilon_n^2$, $\beta_n$, $\sigma_n$, and $\mu_n$. Specifically, we aim to establish the following main properties:
\begin{itemize}
\item[(a)] The sequence $\dfrac{\beta_n}{\epsilon_n^2}$ is both upper and lower bounded.
\item[(b)] There exist a constant $a_0$ (chosen very suitably) such that $\min\{\sigma_n, \mu_n\}\geq a_0 \epsilon_n^2$.
\end{itemize}
Due to the intricate nature and considerable length of the proofs, we will divide the process of establishing properties (a) and (b) into 10 sub-claims.\\
\ \\
\noindent \textbf{Claim 1:} For all $\alpha \in [-\sigma_n,\beta_n]$  
$$\langle J_{\xi_n}(t_n+\alpha), J'_{\eta_n}(t_n+\alpha)\rangle \geq \frac{W_n}{2}\geq \frac{W_0}{4}.$$

\begin{proof}[\emph{\textbf{Proof of Claim 1}}]
%We prove the case $W_0>0$, since the case $W_0<0$ is analogous. 
Note that, 
\begin{eqnarray}\label{EQ1NN}
g_{n}'(t)&=&\langle J_{\xi_n}(t), J'_{\eta_n}(t)\rangle+ \langle J'_{\xi_n}(t), J_{\eta_n}(t)\rangle \nonumber\\
&=& 2\langle J_{\xi_n}(t), J'_{\eta_n}(t)\rangle - W_n \\
&=& W_n + 2 \langle J'_{\xi_n}(t), J_{\eta_n}(t)\rangle. \label{EQ2NN}
\end{eqnarray}
From definition of $\sigma_n$, if $\alpha\in[-\sigma_n,0]$, then $g_n'(t_{n}+\alpha)\geq 0$. Moreover, if $\alpha\in [0,\beta_n]$, then from definition of $J_n(t)$ and (\ref{E4N}) we have 
\begin{eqnarray}\label{EQ3'NN}
||J_{\eta_n}(t_n+\alpha)||&\leq & ||J_{n}(t_n+\alpha)||+||J_{\xi_n}(t_n+\alpha)|| \nonumber \\
&\leq & \epsilon_n^2 + C\lambda^{t_n+\alpha}||\xi_n||.
\end{eqnarray}
Therefore,
\begin{equation}\label{EQ3NN}
|\langle J'_{\xi_n}(t_n+\alpha), J_{\eta_n}(t_n+\alpha)\rangle|\leq C\lambda^{t_n+\alpha}||\xi_n||(\epsilon_n^2 + C\lambda^{t_n+\alpha}||\xi_n||).
\end{equation}
Note that, for $n$ large enough the right side of (\ref{EQ3NN}) converges to $0$, since $\xi_n \to \xi$. Thus, for $n$ large enough, we have that 
\begin{equation}\label{EQ4NN}
|\langle J'_{\xi_n}(t_n+\alpha), J_{\eta_n}(t_n+\alpha)\rangle|\leq \frac{W_0}{4}, \,\, \,\,\ \alpha\in [0,\beta_n].
\end{equation}
Thus, from (\ref{EQ2NN}) and (\ref{EQ4NN}) we can conclude that 
$$g_n'(t_n+\alpha)>0,\,\,\, \alpha\in [0,\beta_n].$$
Consequently, $g_n'(t_n+\alpha)\geq 0$ for all $\alpha\in [-\sigma_n, \beta_n]$, which together with (\ref{EQ0'NN}) and (\ref{EQ1NN}) provides the proof of Claim 1.
\end{proof}

\noindent \textbf{Claim 2:}
There is a constant $\delta_0\in (0,1]$ such that for all $n$ large enough 
\begin{equation}\label{EQ5NN}
\frac{|\langle J_{\xi_n}(t_n+\alpha),J'_{\eta_n}(t_n+\alpha)\rangle|}{||J_{\xi_n}(t_n+\alpha)||\cdot||J'_{\eta_n}(t_n+\alpha)||} \geq \delta_0;\, \,\, \alpha\in [-\sigma_n,\beta_n].
\end{equation}
%The same result holds when $W_0<0$ with $t_n-\alpha$ instead of  $t_n+\alpha$.
\begin{proof}[\emph{\textbf{Proof of Claim 2}}]
For each $\alpha\in [-\sigma_n,\beta_n]$ consider the set 
$$\mathcal{A}_{n,\alpha}:=\Big \{\nu: |\langle\nu, J'_{\eta_n}(t_n+\alpha)\rangle|\geq \frac{W_0}{4}\Big \}.$$
It is easy to see that there is a positive constant $\alpha_{_0}$ such that for every $\nu\in \mathcal{A}_{n,\alpha}$ and all $\alpha\in [-\sigma_n,\beta_n]$ we have 
\begin{equation}\label{EQ5"NN}
|\angle(\nu, J'_{\eta_n}(t_n+\alpha))\pm\frac{\pi}{2}|>\alpha_{_0}>0,
\end{equation}
where $\angle(v,w)$ denotes the angle between $v$ and $w$ (angle between $[-\frac{\pi}{2},\frac{\pi}{2}]$).
Finally, note that 

$$\frac{|\langle J_{\xi_n}(t_n+\alpha),J'_{\eta_n}(t_n+\alpha)\rangle|}{||J_{\xi_n}(t_n+\alpha)||\cdot||J'_{\eta_n}(t_n+\alpha)||}=\cos \angle(J_{\xi_n}(t_n+\alpha), J'_{\eta_n}(t_n+\alpha).$$

\noindent From  Claim 1 $J_{\xi_n}(t_n+\alpha)\in \mathcal{A}_{n,\alpha}$, therefore (\ref{EQ5"NN}) gives us 
$$\cos \angle(J_{\xi_n}(t_n+\alpha), J'_{\eta_n}(t_n+\alpha)\geq \cos \alpha_0: =\delta_0>0,$$
which concludes the proof of Claim 2. 
\end{proof}
\ \\
As an important observation, note that (\ref{EQ5NN}) is invariant for multiplication of a parameter $a>0$, \emph{i.e.},  if we consider the stable and unstable Jacobi field $a J_{\xi_n}(t)$ and $a J_{\eta}(t)$, respectively, then (\ref{EQ5NN}) is valid for the same $\delta_0$.
%%%%%%%%%%%%%%%%%%%%%%%%%%%%%%%%%%%%%%%%%%%%%%%%%%%%%%%%%%%%%%%%NEW
Therefore, for a parameter $\tau>0$, we consider %the stable and unstable vector $\delta\xi$ and $\delta\eta$, respectively.\\
 $$J_{n,\tau}(t)=J_{\tau \eta_n}(t+1)-J_{\delta \xi_n}(t+1)=\tau\, J_n(t).$$
We will choose a suitable parameter $\tau$. In fact: 
note that $$\kappa_{n,\tau}:=||J'_{\tau\eta_{n}}(1)- J'_{\tau\xi_n}(1)||=\tau \kappa_n; \,\,\, \kappa_{0,\tau}:=||J'_{\tau \eta}(1)- J'_{\tau\xi}(1)||=\tau \kappa_0,$$
%$$W_{n, \tau}:=W(J_{\tau\xi_n}, J_{\tau\xi_n})(t)=\tau^2W_{n};\,\,\, W_{0, \tau}:=W(J_{\tau\xi}, J_{\tau\eta})(t)=\tau^2W_{0}$$
Inspired by Lemma \ref{Green-Lemma}, note that 
$$\lim_{\tau\to \infty}\dfrac{\tau^4\kappa_0^4}{\frac{8}{3}\tau^2 \kappa_0^2+\frac{16}{15}k^2}=+\infty.$$
As $\delta_0$ does not depend on $\tau$, then there is $\tau_0>0$ such that
\begin{equation}\label{EQ15*NNN}
\dfrac{\tau_0^4\kappa_0^4}{\Big(\dfrac{8}{3}\tau_0^2\kappa_0^2+\dfrac{16}{15}k^2 \Big)}>\frac{2}{\delta_0^2}+ 5.
\end{equation}

\noindent Using the parameter $\tau_0$ and (\ref{EQ15*NNN}), if necessary, we change $J_{n}$ by $\tau_{_0}J_{n}$ such that, from now on (avoid $\delta_0$) we have 
\begin{equation}\label{EQ15*NN*}
\dfrac{\kappa_0^4}{\Big(\dfrac{8}{3}\kappa_0^2+\dfrac{16}{15}k^2 \Big)}>\frac{2}{\delta_0^2}+ 5.
\end{equation}
The choice of this parameter $\tau_0$ will be used at the end of the proof to obtain a contradiction. \\
\ \\
%%%%%%%%%%%%%%%%%%%%%%%%%%%%%%%%%%%%%%%%%%%%%%%%%%%%%%%%%%%%%%%%%%%%
%%%%%%%%%%%%%%%%%%%%%%%%%%%%%%%%%%%%%%%%%%%%%%%%%%%%%%%%%%%%%%%%%%%%
\noindent \textbf{Claim 3:}
For all $\alpha\in[-\sigma_n,\beta_n]$ 
\begin{equation}\label{EQ6NN}
\frac{(|W_n|-|P_n(\alpha)|)^2}{||J_{\xi_n}(t_n+\alpha)||^2}\leq ||J'_{\eta_n}(t_n+\alpha)||^2\leq \frac{(|W_n|+|P_n(\alpha)|)^2}{\delta_0^2||J_{\xi_n}(t_n+\alpha)||^2},
\end{equation}
where $P_n(\alpha):=\langle J'_{\xi_n}(t_n+\alpha),J_{\eta_n}(t_n+\alpha)\rangle$ and $\delta_0$ as Claim 2.\\
%The same result holds when $W_0<0$ with $t_n-\alpha$ instead of  $t_n+\alpha$.
\begin{proof}[\emph{\textbf{Proof of Claim 3}}]
Using the Wroskian 
\begin{equation}\label{EQ7NN}
\langle J_{\xi_n}(t_n+\alpha),J'_{\eta_n}(t_n+\alpha)\rangle = \langle J'_{\xi_n}(t_n+\alpha),J_{\eta_n}(t_n+\alpha)\rangle +W_n.
\end{equation}
Then (\ref{EQ7NN}) provides
\begin{eqnarray*}
||J_{\xi_n}(t_n+\alpha)||\cdot||J'_{\eta_n}(t_n+\alpha)||&\geq& |\langle J_{\xi_n}(t_n+\alpha),J'_{\eta_n}(t_n+\alpha)\rangle|\\
&\geq & |W_n|-|P_n(\alpha)|,
\end{eqnarray*}
which give us the left side of (\ref{EQ6NN}).\\
The proof of the right side of (\ref{EQ6NN}) follows from the Claim 2 and (\ref{EQ7NN}), since 
\begin{eqnarray*}
||J_{\xi_n}(t_n+\alpha)||\cdot||J'_{\eta_n}(t_n+\alpha)||&\leq& \frac{|\langle J_{\xi_n}(t_n+\alpha),J'_{\eta_n}(t_n+\alpha)\rangle|}{\delta_0}\\
&\leq & \frac{|W_n|+|P_n(\alpha)|}{\delta_0}.
\end{eqnarray*}
\end{proof}
%%%%%%%%%%%%%%%%%%%%%%%%%%%%%%%%%%%%%%%NEW POSITION
\noindent In the next claim, we estimate $\dfrac{\beta_n}{\epsilon_n^2}$. \\%\tcr{Again, the claim holds for $W_0<0$.}\\
\textbf{Claim 4:} The sequence $\Big\{\dfrac{\beta_n}{\epsilon_n^2}\Big\}$ is bounded.
%\begin{equation}\label{EQ8NN}
%$\dfrac{\beta_n}{\epsilon_n^2}
%\end{equation}
\begin{proof}[\emph{\textbf{Proof of Claim 4}}]
Since $f'_n(t_n-1+\beta_n)=0$, using the Taylor's Theorem, for some 
$\tilde{\beta}_{n}\in (t_n-1, t_n-1+\beta_n)$
\begin{eqnarray}\label{EQ9'NN}
f_n(t_n-1)&=& f_n(t_n-1+\beta_n)-f'_n(t_n-1+\beta_n)\beta_n+R(\beta_n) \nonumber\\
\epsilon_n^2&=& f_n(t_n-1+\beta_n)+f''_n(\tilde{\beta}_n)\beta_n^2  \nonumber \\
\epsilon_n^2 &=& f_n(t_n-1+\beta_n)+ 2\Big(||J_{n}'(\tilde{\beta}_n)||^2-K(\tilde{\beta}_n)||J_{n}(\tilde{\beta}_n)||^2\Big )\beta_n^2 ,
\end{eqnarray}
where 
$K(\tilde{\beta}_n)=K(\gamma_{\theta_{n}^1}'(\tilde{\beta}_n), J_{n}(\tilde{\beta}_n))$ is the sectional curvature.
From the last equality 
\begin{equation}\label{EQ9NN}
1=\frac{f_n(t_n-1+\beta_n)}{\epsilon_n^2}+ 2\frac{||J_{n}'(\tilde{\beta}_n)||^2\beta_n^2}{\epsilon_n^2}-2K(\tilde{\beta}_n)||J_{n}(\tilde{\beta_n})||^2\frac{\beta_n^2}{\epsilon_n^2}.
\end{equation}
By hypothesis, we have that $-k^2\leq K(\tilde{\beta}_n)\leq b^2$ for some $k,b>0$. Then. from (\ref{E4N})  
\begin{equation}\label{EQ10'NN}
-k^2 \epsilon_n^2\leq K(\tilde{\beta}_n)||J_{n}(\tilde{\beta_n})||^2\leq b^2\epsilon_n^2.
\end{equation}
From definition of $J_n(t)$, $||J_{n}'(\tilde{\beta}_{n})||$ converges to infinite, so from (\ref{EQ9'NN}) and (\ref{EQ10'NN}) we can conclude that 
$$\lim_{n\to \infty}\beta_n =0.$$
Also, from (\ref{EQ10'NN}) 
\begin{equation}\label{EQ10''NN}
-k^2 \beta_n^2\leq K(\tilde{\beta}_n)||J_{n}(\tilde{\beta_n})||\frac{\beta_n^2}{\epsilon_n^2}\leq b^2\beta_n^2.
\end{equation}
By contradiction, assume that $\Big\{\dfrac{\beta_n}{\epsilon_n^2}\Big\}$ is an unbounded sequence, then since $\beta_n$ converge to $0$, then (\ref{EQ9NN}) and (\ref{EQ10''NN}) provides
\begin{equation}\label{EQ11"NN}
\lim_{n\to \infty}||J_{n}'(\tilde{\beta}_n)||^2\beta_n=0.
\end{equation}

%Since $\Big\{\dfrac{\lambda^{t_n}\beta_n}{\epsilon_n^2}\Big\}$ is an unbounded sequence, then the last equation is equivalent to 
%\begin{equation}\label{EQ11"NN}
%\lim_{n\to \infty}||J_{n}'(\tilde{\beta}_n)||^2\cdot\dfrac{\epsilon_n^2}{\lambda^{t_n}}=0.
%\end{equation}

\noindent Put $\bar{\beta}_{n}\in [0,\beta_n]$ such that $\tilde{\beta}_{n}=t_n-1+\bar{\beta}_n$. Since $J_{\xi_n}(t)$  is a stable Jacobi field, then the definition of $J_n(t)$ and (\ref{EQ11"NN}) give us 
\begin{equation}\label{EQ11'NN}
\lim_{n \to \infty} ||J_{\eta_n}'(t_n+ \bar{\beta}_n)||^2 \beta_n = 0.
\end{equation} From (\ref{EQ6NN}) 
\begin{equation}\label{EQ8NN'}
\frac{(|W_n|-|P_n(\bar{\beta}_{n})|)^2}{||J_{\xi_n}(t_n+\bar{\beta}_{n})||^2}\leq ||J'_{\eta_n}(t_n+\bar{\beta}_{n})||^2,
\end{equation}
where $P_n(\bar{\beta}_n):=\langle J'_{\xi_n}(t_n+\bar{\beta}_n),J_{\eta_n}(t_n+\bar{\beta}_n)\rangle$.\\
%Multiplying (\ref{EQ8NN'}) by  $\dfrac{\epsilon_n^2}{\lambda^{t_n}}$ 
From mean value theorem, for some  $w\in [t_n, t_n+\bar{\beta}_n]$

\begin{equation}\label{EQ1"NN}
\frac{||J_{\xi_n}(t_n+\bar{\beta}_n)||^2}{\beta_n}=2\langle J'_{\xi_n}(w),J_{\xi_n}(w) \rangle\frac{\bar{\beta}_n}{\beta_n} + \frac{\epsilon_n^2}{\beta_n}.
\end{equation}
As $|\langle J'_{\xi_n}(w),J_{\xi_n}(w) \rangle|\leq C^2\lambda^{2w}||\xi_n||\leq  C^2\lambda^{2t_n}||\xi_n||$, $\dfrac{\bar{\beta}_n}{\beta_n}\leq 1$, and $\dfrac{\epsilon_n^2}{\beta_n} \to 0$ (because, we are assuming assume that $\Big\{\dfrac{\beta_n}{\epsilon_n^2}\Big\}$ is an unbounded sequence), then 
from (\ref{EQ1"NN}) we have 

$$\ds\lim_{n \to \infty}\frac{\beta_n}{||J_{\xi_n}(t_n+\bar{\beta}_n)||^2}=\infty.$$
Finally, as $W_n$ converges to $W_0>0$ and $P_n(\bar{\beta}_n)$ converges to $0$, then from the last inequality and (\ref{EQ8NN'}) we conclude that $\ds\lim_{n\to \infty}||J_{\eta_n}'(t_n+\bar{\beta}_n)||^2\beta_n=\infty$, which is a contradiction with (\ref{EQ11'NN}) and consequently the sequence $\Big\{\dfrac{\beta_n}{\epsilon_n^2}\Big\}$ is a bounded sequence.
\end{proof}

The claim below holds for any positive constant $a_0 > 0$. However, we select a specific value for $a_0$ to derive a contradiction. Specifically, we set $a_{0}:=\dfrac{1-\delta_0^2}{W_0}+2$, where $\delta_0$ is as defined in claim 3.
\\
%%%%%%%%%%%%%%%%%%%%%%%%%%%%%NEW POSITION
\noindent \textbf{Claim 5:} For all $a_0>0$, we have 
$$\ds\lim_{n\to \infty}\Big |\dfrac{||J_{\xi_n}(t_n)||^2}{||J_{\xi_n}(t_n+\alpha)||^2}-1\Big | = 0, \,\, \text{for all}\, \, \alpha\in[-a_0\epsilon_n^2,\beta_n].$$
\begin{proof}[\emph{\textbf{Proof of Claim 5}}] Consider $\alpha\in [-a_0\epsilon_n^2,\beta_n]$, then from mean value theorem (similar to (\ref{EQ1"NN}))
\begin{equation}\label{EQ1'NN}
\frac{||J_{\xi_n}(t_n+\alpha)||^2}{||J_{\xi_n}(t_n)||^2}-1=2\frac{\langle J'_{\xi_n}(w),J_{\xi_n}(w) \rangle}{||J_{\xi_n}(t_n)||^2} \alpha, \,
\end{equation}
for some  $w\in [t_n, t_n+\alpha]$ or $w\in [t_n+\alpha, t_n]$.\\
As $\epsilon_n=||J_{\xi_n}(t_n)||$, then  % and $\frac{\alpha}{\epsilon_n^2}\leq \frac{\beta_n}{\epsilon_n^2}$, then 
\begin{eqnarray}\label{EQ2'NN}
\dfrac{|\langle J'_{\xi_n}(w),J_{\xi_n}(w)\rangle\alpha|}{\epsilon_n^2}&\leq & ||J'_{\xi_n}(w)||\cdot||J_{\xi_n}(w)||\Big \{a_0, \frac{\beta_n}{\epsilon_n^2} \Big \} \nonumber\\
&\leq &  C^2\lambda^{2w}||\xi_n||\max \Big \{a_0, \frac{\beta_n}{\epsilon_n^2} \Big \}.
\end{eqnarray}
From Claim 4 the sequence $\{\frac{\beta_n}{\epsilon_n^2}\}$ is bounded, then
the right of (\ref{EQ2'NN}) converges to $0$, which from (\ref{EQ1'NN}) completes the proof of claim.
%for $n$ large enough.
%Thus $\dfrac{||J_{\xi_n}(t_n-\alpha)||^2}{||J_{\xi_n}(t_n)||^2}>\dfrac{9}{10}$.
\end{proof}

%%%%%%%%%%%%%%%%%%%%%%%%%%%%%%%%%%%%%NEW POSITION

%\noindent A suitable parameter $a_0$ will be chosen after Claim 10.\\
\noindent With the assistance of the last claims, we establish a precise estimate of $\frac{\beta_n}{\epsilon_n^2}$.\\

\noindent \textbf{Claim 6:} For all $\epsilon>0$ there is $n_0$ such that $$\frac{\delta_0^2}{W_0}-\epsilon < \dfrac{\beta_n}{\epsilon_n^2}< \frac{1}{W_0}+\epsilon, \,\,\,\, \text{for all} \,\,\, n\geq n_0,$$

\begin{proof}[\emph{\textbf{Proof of Claim 6}}]
From (\ref{E3N}),  $f'_{n}(t_n-1+\beta_n)=0$, then
\begin{eqnarray*}%\label{E5N}
|f'_{n}(t_n-1+\beta_n)-f'_{n}(t_n-1)|&=&|f''_n(\tilde{\beta}_n)|\beta_n \nonumber\\|f'_{n}(t_n-1)|&=&  2\Big|||J_{n}'(\tilde{\beta}_n)||^2-K(\tilde{\beta}_n)||J_{n}(\tilde{\beta}_n)||^2\Big|\cdot \beta_n
\end{eqnarray*}
for some $\tilde{\beta}_n\in (t_n-1,t_n-1+\beta_n)$, where  %or $\tilde{\beta}\in (t_n-1,\beta_n)$ and 
 $K(\tilde{\beta}_n)=K(\gamma_{\theta_{n}^{1}}'(\tilde{\beta}_n), J_{n}(\tilde{\beta_n}))$ is the sectional curvature. \\
The last equation and  (\ref{E1N}) provides 
\begin{equation}\label{EQ10NN}
 2|\langle J'_{\xi_n}(t_n), J_{\xi_{n}}(t_n)\rangle-W_n|=2\Big|||J_{n}'(\tilde{\beta}_n)||^2-K(\tilde{\beta}_n)||J_{n}(\tilde{\beta}_n)||^2\Big|\cdot \beta_n.
\end{equation}

\noindent Note that the left side of (\ref{EQ10NN}) converges to $2W_0$. Moreover, 
since \newline $-k^2\leq K(\tilde{\beta}_n)\leq b^2$, then from (\ref{E4N})  
$$-k^2 \epsilon_n^2\leq K(\tilde{\beta}_n)||J_{n}(\tilde{\beta}_n)||\leq b^2\epsilon_n^2.$$
Consequently, from (\ref{EQ10NN}) 
\begin{equation}\label{EQ11NN}
\lim_{n \to \infty} ||J_{n}'(\tilde{\beta}_n)||^2\cdot \beta_n = W_0.
\end{equation}
Put $\tilde{\beta}_n:=t_n-1+ \bar{\beta}_n$, for some $\bar{\beta}_n\in [0,\beta_n]$. Since $J_{\xi_n}(t)$  is a stable Jacobi field, then the definition of $J_n(t)$ and (\ref{EQ11NN}) give us 
\begin{equation}\label{EQ11'NN'}
\lim_{n \to \infty} ||J_{\eta_n}'(t_n+ \bar{\beta}_n)||^2\cdot \beta_n = W_0.
\end{equation}

\noindent Finally, from Claim 3 or (\ref{EQ6NN}) we have 
\begin{equation}\label{EQ12NN}
\frac{(|W_n|-|P_n(\bar{\beta}_n)|)^2\epsilon_n^2}{||J_{\xi_n}(t_n+\bar{\beta}_n)||^2}\leq ||J'_{\eta_n}(t_n+\bar{\beta}_n)||^2\epsilon_n^2\leq \frac{(|W_n|+|P_n(\bar{\beta}_n)|)^2\epsilon_n^2}{\delta_0^2||J_{\xi_n}(t_n+\bar{\beta}_n)||^2}.
\end{equation}
Similar to (\ref{EQ3NN}),
$\ds\lim_{n\to \infty}|P_n(\bar{\beta}_n)|=0$. Thus, claim 5 provides that the left side and right side of (\ref{EQ12NN}) converge to $W_0^2$ and $\frac{W_0^2}{\delta_0^2}$, respectively. 
Therefore, from (\ref{EQ11'NN'}) and (\ref{EQ12NN}), given $\epsilon>0$ we can find $n_0$ large enough such that such that 
$$\frac{\delta_0^2}{W_0}-\epsilon < \dfrac{\beta_n}{\epsilon_n^2}< \frac{1}{W_0}+\epsilon, \,\,\,\, \text{for all} \,\,\, n\geq n_0,$$
and the proof of claim is complete.
\end{proof}
%%%%%%%%%%%%%%%%%%NEW LEMMA
\noindent Consider the function $h_{n}(t): =||J_{\eta_n}(t)||^2$, and define 
$$\mu_n:=\sup_{\alpha\in (0,t_n]}\{h'_{n}(t)\leq 0;\,\, \, t\in [t_n-\alpha, t_n]\}.$$
The number $\mu_n$ is well defined since $J_{\eta_n}(t)$ is a unstable Jacobi field and $h_{n}(t_n)=0$.\ \\
\\
\noindent \textbf{Claim 7:} Consider  $\Gamma_n(\alpha):=||J_{\eta_n}(t_n+\alpha)||^2$, then there is a constant $\Gamma_0$ such that, for $n$ large enough
$$\Gamma_n(\alpha)\leq \Gamma_0, \,\,\ \text{for all} \,\, \alpha \in [-\chi_n, \beta_n],$$
where $\chi_n:=\min\{\sigma_n, \mu_n, a_0\epsilon_n^2\}$ and $a_0$ as Claim 5.
\begin{proof}[\emph{\textbf{Proof of Claim 7}}]
From (\ref{EQ3'NN}) we have that for all $\alpha\in [0,\beta_n]$
\begin{equation*}
||J_{\eta_n}(t_n+\alpha)||\leq \epsilon_n^2+ C\lambda^t_n||\xi_n||\leq 1
\end{equation*}
for $n$ large enough.\\
So, from now on, we only need to be worried about $\alpha\in [-\chi_n,0]$. Consider $\alpha\in [-\chi_n,0]$, since $h(t_n)=0$, then from Mean Value Theorem
$$-\Gamma_n(\alpha)= -h_n(t_n+\alpha) = h_n'(w)\alpha, $$
for some $w\in (t_n+\alpha, t_n)$. Also,
$$h_{n}'(t_n)-h_{n}'(w)=h_n''(\tilde{w})(t_n-w)=2\Big(-K(\tilde{w})||J_{\eta_n}(\tilde{w})||^2+||J_{\eta_n}'(\tilde{w})||^2\Big )(t_n-w),$$
for some $\tilde{w}\in (w,t_n)$. Here $K(\tilde{w})=K(\gamma_{\theta_{n}}'(\tilde{w}), J_{\eta_n}(\tilde{w}))$ is the sectional curvature. \\
As $-\alpha \leq \chi_n\leq \mu_n$, then $t_n-\mu_n\leq t_n+\alpha<w<\tilde{w}<t_n$. Therefore, the definition of $\mu_n$ provides 
$||J_{\eta_n}(\tilde{w})||^2\leq \Gamma_n(\alpha)$.
Moreover, since $K(\tilde{w})\geq -k^2$ and $h_{n}'(t_n)=0$, we obtain
\begin{eqnarray}\label{EQ1N*}
\Gamma_n(\alpha)&\leq &  -2\Big(-k^2||J_{\eta_n}(\tilde{w})||^2+||J_{\eta_n}'(\tilde{w})||^2\Big )\alpha \nonumber \\
&\leq& -2\Big(-k^2\Gamma_n(\alpha)+||J_{\eta_n}'(\tilde{w})||^2\Big )\alpha \nonumber \\
&\leq & 2a_0\Big(-k^2\Gamma_n(\alpha)+||J_{\eta_n}'(\tilde{w})||^2\Big )\epsilon_n^2,
\end{eqnarray}
where in the last inequality we use that $-\alpha\leq \chi_n\leq a_0\epsilon_n^2$. \\
%\tcr{PAREI AQUI 11/10/2023 - 23:33}\\
\ \\
Therefore, our task now is to estimate $||J_{\eta_n}'(\tilde{w})||^2\epsilon_n^2$, for  $\tilde{w}\in [t_n+\alpha, t_n]$ and $\alpha\geq -\chi_n$. %Since $\alpha \leq \epsilon_n^2$, then it is enough to estimate 
%$$||J_{\eta_n}'(t_n-\alpha)||^2\epsilon_n^2, \,\,\, \text{for}\,\,\, \alpha\in[0,\beta_n].$$
We can write $\tilde{w}=t_n+a$, with $a\geq \alpha$. Thus, as $a\geq \alpha\geq -\chi_n\geq -a_0\epsilon_n^2$, then the Claim 5 provides 
\begin{equation}\label{EQ2'N}
\lim_{n\to \infty}\frac{\epsilon_n^2}{||J_{\xi_n}(t_n+a)||}=1.
\end{equation}
Moreover, as $a\geq \alpha \geq - \chi_n \geq -\sigma_n$, then the Claim 3 establish that  
\begin{equation}\label{EQ3'N}
||J'_{\eta_n}(t_n+a)||^2\epsilon_n^2\leq \frac{(W_n+|P_n(a)|)^2\epsilon_n^2}{\delta_0^2||J_{\xi_n}(t_n+a)||^2}.
\end{equation}
As $W_n \to W_0$ and $P_n(a)\to 0$, then from (\ref{EQ2'N}) the right side of (\ref{EQ3'N}) converges to $\dfrac{W_0^2}{\delta_0^2}$.
Consequently, for $n$ large enough, 

$$||J_{\eta_n}'(\tilde{w})||^2\epsilon_n^2=||J'_{\eta_n}(t_n+a)||^2\epsilon_n^2\leq \dfrac{W_0^2}{\delta_0^2}+1.$$
Then, from (\ref{EQ1N*}) we conclude that 

$$\Gamma_n(\alpha)\leq \frac{W_0^2+\delta_0^2}{(1+2a_0k^2)\delta_0^2}:=\Gamma_0.$$

\end{proof}

\noindent \textbf{Claim 8:} For $n$ large enough $$\mu_n\geq \min\{a_0\epsilon_n^2, \sigma_n\}.$$ 
\begin{proof}[\emph{\textbf{Proof of Claim 8}}]
Similar to last arguments, as $h_n'(t_n)=0$ and \newline $h_n'(t_n-\mu_n)=0$, then  
$$0=h'(t_n)-h_n'(t_n-\mu_n)=2\Big(-K(\omega_n)||J_{\eta_n}(\omega_n)||^2+||J_{\eta_n}'(\omega_n)||^2\Big )\mu_n$$
for some $\omega_n\in [t_n-\mu_n, t_n]$ and again $K(w_n)=K(\gamma_{\theta_{n}}'(w_n), J_{\eta_n}(w_n))$ is the sectional curvature.
 Thus 
$$b^2\geq K(\omega_n)=\dfrac{||J_{\eta_n}'(\omega_n)||^2}{||J_{\eta_n}(\omega_n)||^2}.$$
The definition of Anosov flow and definition of $\mu_n$ provide
\begin{equation}\label{EQ7N}
C^{-2}\lambda^{-2\omega_n}||\eta_n||\leq \Big(\dfrac{||J_{\eta_n}'(\omega_n)||^2}{||J_{\eta_n}(\omega_n)||^2}+1\Big)||J_{\eta_n}(\omega_n)||^2\leq (b^2+1)\Gamma_{n}(\mu_n).
\end{equation}
Now, assume by contradiction that $\mu_n<\min\{a_0\epsilon_n^2, \sigma_n\}$, then $\mu_n=\chi_n$. Thus, Claim 7 and  (\ref{EQ7N}) give us 
$$\dfrac{C^{-2}\lambda^{-2\omega_n}||\eta_n||}{b^2+1}\leq \Gamma_0.$$
Since $\eta_n$ converges to $\eta$, in particular, $||\eta_n||$ is bounded, then we have a contradiction, since for $n$ large enough the left side of the last inequality goes to infinite.  
\end{proof}
%%%%%%%%%%%%%%%%%%%%%%%%%%%%%%%%%%%%%%%%%%%%%%%%%%%%%%%%%%%%%%%%%%%%%%%%%
\noindent \textbf{Claim 9:} For $n$ large enough
$$\sigma_n\geq \min\{\mu_n, a_0\epsilon_n^2\}.$$ 
\begin{proof}[\emph{\textbf{Proof of Claim 9}}]
Recall that $g_n(t)=\langle J_{\xi_n}(t), J_{\eta_n}(t)\rangle$ and whose derivative is  

$$g_n'(t)=2\langle J_{\xi_n}(t), J'_{\eta_n}(t)\rangle - W_n$$
By definition of $\sigma_n$, if $g'_n(t_n-\sigma_n)\neq 0$, then $\mu_n=t_n$ and the claim holds. Thus, we can assume $g'_n(t_n-\sigma_n)=0$ and then   
$$\langle J_{\xi_n}(t_n-\sigma_n), J'_{\eta_n}(t_n-\sigma_n)\rangle=\frac{W_n}{2}.$$
Moreover, the Wroskian satisfies 
$$\langle J_{\xi_n}(t), J'_{\eta_n}(t)-\langle J'_{\xi_n}(t), J_{\eta_n}(t)\rangle=W_n.$$
Consequently, 
$$\langle J'_{\xi_n}(t_n-\sigma_n), J_{\eta_n}(t_n-\sigma_n)\rangle=-\frac{W_n}{2}.$$
Assume by contradiction that  $\sigma_n<\min\{\mu_n, a_0\epsilon_n^2\}$, then $\sigma_n=\chi_n$. Thus, the last equation, Claim 7, and  (\ref{EQ0'NN}) give us 
\begin{eqnarray}\label{EQ8N}
\frac{W_0^2}{16}\leq \frac{W_n^2}{4} &\leq & ||J'_{\xi_n}(t_n-a_n)||^2\cdot ||J_{\eta_n}(t_n-\sigma_n)||^2 \nonumber \\ 
&\leq & C^2\lambda^{2(t_n-a_n)}\Gamma_{n}(-\sigma_n)\leq C^2\lambda^{2(t_n-\sigma_n)}\Gamma_0.
\end{eqnarray}
The right side of (\ref{EQ8N}) converges to $0$, which is a contradiction since $W_0>0$.\newline
\end{proof}

\noindent Since $a_n$, $\mu_n$, and $\epsilon_n^2$ are non-negative numbers, then as an immediate consequence of Claim 8 and Claim 9 we have\\

\noindent \textbf{Claim 10:} For $n$ large enough
$$\min\{\sigma_n,\mu_n\}\geq a_0\epsilon_n^2,$$
consequently, $\chi_n=a_0\epsilon_n^2$. 
%If the sectional curvature is bounded above, then

%%%%%%%%%%%%%%%%%%%%%%%%%%%%%%%%%%%%%%%%%%%%%%%%%%%%%%%%%%%%%%%%%%%%%%%%%
%\ \\
%\ \\
%\ \\
%Also, from (\ref{E1N}), $||J_{n}'(\tilde{\beta})||$ converges to infinite. Therefore (\ref{EQ9NN}) ensures that 
%$$\lim_{n\to \infty}\beta_n=0.$$
\ \\
%%%%%%%%%%%%%%%%%%%%%%%%%%%%%%%%%5NEW ARGUMENT
Continuing with the proof of the lemma, from the Claim 6, for $n$ large enough 
$$-\omega_n:=\beta_n - \Big(\frac{1}{W_0} + 1 \Big)\epsilon_n^2<0.$$

\noindent Remember the parameter $a_{_0}:=\dfrac{1-\delta_0^2}{W_0}+2$. From Claim 6, it is easy to see that 
$$a_{_0}\epsilon_n^2 \geq \Big(\frac{1}{W_0} +1 \Big)\epsilon_n^2-\beta_n =\omega_n\geq \frac{1}{2}\epsilon_n^2.$$

\noindent Therefore, 
\begin{equation}\label{EQ13'NN'}
t_n-a_0\epsilon_n^2<t_n - \omega_n<t_n<t_n+\tilde{\omega}_n< t_n+\beta_n.
\end{equation}
%if we call $\omega_n:=(\frac{1}{W_0}+1)\epsilon_n^2-\beta_n>0$, then 
%$$t_n-1-\delta_n<t_n-1.$$
Moreover, we know that $\gamma_{\theta^1_n}(t)$ has no conjugate point in $[-1, t_n-1]$. In particular, for all $\vartheta>0$ small enough, $\gamma_{\theta^1_n}(t)$ has no conjugate point in $[-1, t_n-1-\omega_n-\vartheta]$. Thus, since $J_n(0)=0$, then the  Lemma \ref{Green-Lemma} provides that  for all $\vartheta>0$ small enough 
 
$$||J_{n}(t_n-1 - \omega_n)||^2\geq \rho(\omega_n-\vartheta).$$ 
Since $\vartheta$ is arbitrary, then there we have 
$$||J_{n}(t_n-1-\omega_n)||^2\geq \rho(\omega_n).$$ 
The last inequality and definition of $\rho(t)$ provides 
\begin{equation}\label{EQ13'NN}
\Big(k\coth k+\dfrac{1}{\omega_n}+\dfrac{k^2}{3}\omega_n\Big)||J_{n}(t_n-1-\omega_n)||^2\geq \dfrac{\kappa_n^4}{\Big(\dfrac{8}{3}\kappa_n^2+\dfrac{16}{15}k^2 \Big)},
\end{equation}
where $\kappa_n=||J'_n(0)||=||J'_{\eta_n}(1)-J'_{\xi_n}(1)||$.\\
Multiplying and dividing by $\epsilon_n^2$ the left side of (\ref{EQ13'NN}) %and remembering that $||J_n(t_n-1)||=\epsilon_n$, then 
\begin{equation}\label{EQ14'NN}
\dfrac{\kappa_n^4}{\Big(\dfrac{8}{3}\kappa_n^2+\dfrac{16}{15}k^2 \Big)\Big(\epsilon_n^2\, k\coth k +\dfrac{\epsilon_n^2}{\omega_n}+\dfrac{k^2}{3}\epsilon_n^2\omega_n\Big)}\leq \dfrac{||J_{n}(t_n-1-\omega_n)||^2}{\epsilon_n^2}.
\end{equation}
%Note that, given $\epsilon>0$ small, then there is $n_0$ such that, for all $n\geq n_0$
%$$\dfrac{\omega_n}{\epsilon_n^2}=\frac{1}{W_0}+1-\frac{\beta_n}{\epsilon_n^2}>(1-\epsilon).$$
From Claim 6, let $\epsilon>0$, such that for $n$ large enough we have that $$\epsilon_n^2\, k\coth k +\dfrac{k^2}{3}\epsilon_n^2\omega_n<\frac{\epsilon}{1-\epsilon} \,\,\, \text{and} \,\,\, \frac{\omega_n}{\epsilon_n^2}>1-\epsilon$$
Therefore, 
$$\dfrac{1}{\Big(\epsilon_n^2\, k\coth k +  \dfrac{\epsilon_n^2}{\omega_n}+\dfrac{k^2}{3}\epsilon_n^2\omega_n\Big)}>\frac{1-\epsilon}{1+\epsilon}>\frac{1}{2},$$
consequently, (\ref{EQ14'NN}) become 
\begin{equation}\label{EQ15*NN}
\frac{1}{2}\cdot\dfrac{\kappa_n^4}{\Big(\dfrac{8}{3}\kappa_n^2+\dfrac{16}{15}k^2 \Big)}\leq \dfrac{||J_{n}(t_n-1-\omega_n)||^2}{\epsilon_n^2}.
\end{equation}
%\tcr{PAREI AQUI}
%%%%%%%%%%%%%%%%%%%%%%%%%%%%%%%%%%%%%%%%%%%%%%%

\noindent From Taylor's theorem centered in $t_n-1+\beta_n$ we can write (similar to (\ref{EQ9'NN}))
\begin{eqnarray}\label{EQ15'NN'}
\frac{f_n(t_n-1-{\omega}_n)}{\epsilon_n^2}&=& \frac{f_n(t_n-1+\beta_n)}{\epsilon_n^2} - f'_n(t_n-1+\beta_n)\frac{(\beta_n+\omega_n)}{\epsilon_n^2} \nonumber \\
&+& 2\Big(||J_{n}'({\alpha}_n)||^2-K({\alpha}_n)||J_{n}({\alpha}_n)||^2\Big )\frac{(\beta_n+{\omega}_n)^2}{\epsilon_n^2},
\end{eqnarray}
for some $\alpha_n\in [t_n-1-{\omega}_n, t_n-1+\beta_n]$.

Let us estimate the right side of the last equation.
First note that from (\ref{EQ9NN}) and (\ref{EQ10'NN}) we have that for $n$ large enough 
\begin{equation}\label{EQ15NN}
\frac{f_n(t_n-1+\beta_n)}{\epsilon_n^2}\leq \frac{3}{2}.
\end{equation}

\noindent From (\ref{EQ13'NN'}), $\omega_n\leq a_0\epsilon_n^2$, thus since $\alpha_n\in [t_n-1-{\omega}_n, t_n-1+\beta_n]$, then Claim 7 and Claim 10 provides   

\begin{equation*}\label{EQ16'NN}
-k^2 \Gamma_0\leq K(\alpha_n)||J_{n}(\alpha_n)||\leq b^2\Gamma_0.
\end{equation*}
Also \begin{equation*}
\frac{(\beta_n+\omega_n)^2}{\epsilon_n^2}=\Big(\frac{1}{W_0}+\epsilon \Big)^2\epsilon_n^2.
\end{equation*}
Thus, it is easy to see that 
\begin{equation}\label{EQ16NN}
\lim_{n\to \infty} K(\alpha_n)||J_{n}(\alpha_n)||\frac{(\beta_n+\omega_n)^2}{\epsilon_n^2}=0.
\end{equation}
Moreover, from Claim 3 and Claim 5, for $n$ large enough 
$$||J'_{\eta_n}(\alpha_n)|| \frac{(\beta_n+\omega_n)^2}{\epsilon_n^2}=||J'_{\eta_n}(\alpha_n)||\epsilon_n^2\Big(\frac{1}{W_0}+\epsilon \Big)^2\leq \frac{1}{\delta_0^2}+\frac{1}{8}.$$
Thus, from definition of $J_n(t)$ the last equation implies that for $n$ large enough 
\begin{equation}\label{EQ17NN}
||J'_{n}(\alpha_n)|| \frac{(\beta_n+\omega_n)^2}{\epsilon_n^2}\leq \frac{1}{\delta_0^2}+\frac{1}{4}.
\end{equation}
Since $f_n'(t_n-1+\beta_n)=0$, then  (\ref{EQ15'NN'}), (\ref{EQ15NN}), (\ref{EQ16NN}), and (\ref{EQ17NN}) allow us to conclude that  for $n$ large enough 
\begin{equation}\label{EQ18NN}
\frac{f_n(t_n-1-{\omega}_n)}{\epsilon_n^2}\leq \frac{1}{\delta_0^2}+ 2.
\end{equation}

\noindent Finally, as $f_n(t_n-1-{\omega}_n)=||J_n(t_n-1-{\omega}_n)||^2$, then (\ref{EQ15*NN}) and (\ref{EQ18NN}) provide 
\begin{equation*}\label{EQ15'NN}
\dfrac{\kappa_n^4}{\Big(\dfrac{8}{3}\kappa_n^2+\dfrac{16}{15}k^2 \Big)}\leq \frac{2}{\delta_0^2}+ 4,
\end{equation*}
which implies that 
\begin{equation}\label{EQ15''NN}
\dfrac{\kappa_0^4}{\Big(\dfrac{8}{3}\kappa_0^2+\dfrac{16}{15}k^2 \Big)}\leq \frac{2}{\delta_0^2}+ 4.
\end{equation}
It is clear that (\ref{EQ15''NN}) is a contradiction with (\ref{EQ15*NN*}) and therefore we conclude the proof of Claim and consequently conclude the proof of Lemma.
\end{proof} 
\end{proof}

\begin{R}\label{RNEW1}
In the proof of \emph{Lemma \ref{Bounded t_{+}}}, only the upper bound of the curvature was utilized in \emph{(\ref{EQ10'NN})}, \emph{(\ref{EQ10NN})}, and \emph{(\ref{EQ7N})}. Therefore, we can amend the condition of bounded curvature from above to the following condition:
Put $$\mathcal{M}_{n} := \max \Big \{\sup_{t\in [t_n-1, t_n-1+\beta_n]} K(J_{n}(t), \gamma_{\theta_n^1}(t)), \sup_{t\in [t_n, t_n+\beta_n]} K(J_{\eta_n}(t), \gamma_{\theta_n}(t))\Big\},$$
then
\begin{equation}\label{EQ1N}
\liminf_{n\to \infty} \lambda^{t_n}\mathcal{M}_n =0,
\end{equation}
where $\lambda$ is the constant of contraction of the definition of Anosov flow.\\
Since $0<\lambda<1$, then it is clear that if $M$ has curvature bounded above, then satisfies \emph{(\ref{EQ1N})}.
\end{R}

To conclude this section, we establish Lemma \ref{Bounded t_{+}} in the two-dimensional scenario without imposing any restriction on the upper limit of the sectional curvature. To accomplish this, we require the following lemma.

\begin{lemma}[\textbf{Zeros of Jacobi Fields in Dimension 2}]\label{LN1}
If $a<b$ and $J$ is a Jacobi field such that $J(a)=J(b)=0$, then any Jacobi field $\tilde{J}$ has a zero in $[a,b]$. Moreover, if $\tilde{J}$ and $J$ are linearly independent, then $\tilde{J}$ has a zero in $(a,b)$.
\end{lemma}

\begin{lemma}\label{Bounded t_{+}-dim 2}
If the sectional curvature of $M$ is bounded below, then the sets $\mathcal{B}^{u}_{+}$ and $\mathcal{B}^{s}_{-}$ are closed.
\end{lemma}

The proof follows the same approach as the proof of Lemma \ref{Bounded t_{+}}, but we must avoid relying on the condition of the curvature being bounded above. Consequently, we only establish which claims of the proof of Lemma \ref{Bounded t_{+}} remain valid in the two-dimensional case without an upper bound on the sectional curvature.
\begin{proof}[\emph{\textbf{Proof of Lemma \ref{Bounded t_{+}-dim 2}}}]
Maintain the notation of Lemma \ref{Bounded t_{+}} throughout. For each $n$, let $V_n(t)$ be a parallel vector field along $\gamma_{{\theta_n}}(t)$. If $J$ is a Jacobi field along $\gamma{{\theta_n}}(t)$, then $J(t)=f(t)V_n(t)$, where $f(t)$ is a real function. Therefore, all Jacobi fields along $\gamma{{\theta_n}}(t)$ can be regarded as real functions. Henceforth, let $J{\xi_n}$ and $J_{\eta_n}$, defined similarly to Lemma \ref{Bounded t_{+}}, be considered as real functions. In this case, without loss of generality, we can assume that $J_{\eta_n}>0$ in $[0,t_n)$. From Lemma \ref{First-Time} and Lemma \ref{LN1}, we have that $J_{\xi_n}(t)$ has a unique zero $r_n$ in $[0,t_n)$. Consequently, $W_n=J_{\xi_n}(t_n)J_{\eta_n}'(t_n)>0$ and $W_0>0$.\\
\ \\
\textbf{Claim 1:} There is $u_n>t_n-1$ such that $J_n(u_n)=0$.
\begin{proof}[\emph{\textbf{Proof of Claim 1}}]

%Moreover, as $J_{\xi_n}(t)$ is stable and \\
%%%%%%%%%%%%%%%%%%%%%%%%%%%%%%%%%%%%%%%
If $J_{\xi_n}(t+1)$ has a zero $z_n$, for some $z_n-1>t_n-1$, then from Lemma \ref{LN1} $J_{xi_n}(t)$ has a zero in $(t_n,z_n)$ and consequently $J_{n}$ has a zero in $(t_n-1, z_-1)$. Therefore, we can assume that and $J_{\eta_n}(t+1)<0$ for $t>t_n-1$. \\
On the other hand, by the uniqueness  of $r_n$ we know that $$J_n(t_n-1)=J_{\eta_n}(t_n)-J_{\xi_n}(t_n)=-J_{\xi_n}(t_n)>0.$$

Finally, we can assume by contradiction that $J_{n}(t)\neq 0$ for all $t>t_n-1$. Then from the last inequality  $J_{n}(t)>0$ and then 
$$J_{\xi_n}(t)<J_{\eta_n}(t)<0, \,\,\, t>t_n,$$
which allows us to conclude that 
$$|J_{\eta_n}(t)|<|J_{\xi_n}(t)|, \,\,\, t>t_n.$$
The last inequality provides a contradiction since $J_{\xi_n}$ is a stable Jacobi field and $J_{\eta_n}$ is unstable.
\end{proof}
\noindent Using the Claim 1, we consider the positive number 

$$\beta_n:=\inf \{u>0: J_{n}(t_n-1+u)=0\}.$$
This number $\beta_n$ has the same role as the $\beta_n$ defined in the proof of Lemma \ref{Bounded t_{+}}, but with the additional and important  properties $J_{n}(t_n-1+\beta_n)=0$.\\
To conclude the proof of this lemma, we need to prove that all ten claims of the proof of Lemma are valid without the condition of curvature bounded above. It is not difficult to check, such condition on the curvature is only used in Claim 4, Claim 6, and Claim 8. So, from now on we focus on proving such three Claims. 
Thus, note that in the two-dimensional case, the number $\delta_0$ of Claim 3 is equal to $1$.\\
\indent We also note that Claim 5 and Claim 6 of the proof of the last lemma depend on Claim 4, but here we joined both Claim 4 and Claim 6 in a single claim, and consequently Claim 5 is already valid. \\
\textbf{Claim 2:} In this case $\ds\lim_{n\to \infty}\dfrac{\beta_n}{\epsilon_n^2}=\dfrac{1}{W_0}$.
\begin{proof}[\emph{\textbf{Proof of Claim 2}}]
First, we prove that the sequence $\Big\{\dfrac{\beta_n}{\epsilon_n^2}\Big\}$ is bounded. For this sake, note that $J_{n}(t_n-1+\beta_n)=0$, so the mean value theorem provides
$$-\dfrac{\epsilon_n}{\beta_n}=\dfrac{J_n(t_n-1+\beta_n)-J_{n}(t_n-1)}{\beta_n}=J_{n}'(\tilde{\beta}_n),$$
for some $\tilde{\beta}_n\in (t_n-1, t_n-1+\beta_n)$. Therefore 
\begin{equation}\label{EQNEW1}
\dfrac{\epsilon_n^2}{\beta_n}=|J_{n}'(\tilde{\beta}_n)|^2\cdot\beta_n.
\end{equation}
If $\Big\{\dfrac{\beta_n}{\epsilon_n^2}\Big\}$ is an unbounded sequence, then passing to a sub-sequence if necessary, we have that $\ds\lim_{n\to \infty}\dfrac{\epsilon_n^2}{\beta_n}=0$, and the last inequality gives us that 
$$\lim_{n\to \infty}|J_{n}'(\tilde{\beta}_n)|^2\cdot\beta_n=0.$$
Hence, following the same lines of (\ref{EQ11'NN}), (\ref{EQ8NN'}), and (\ref{EQ1"NN}) we find a contradiction and the sequence $\Big\{\dfrac{\beta_n}{\epsilon_n^2}\Big\}$ should be bounded, and consequently, Claim 5 of the proof of last lemma holds.\\
Now, writing $\tilde{\beta}_n=t_n-1+\bar{\beta}_n$, for some $\bar{\beta}_n\in [0,\beta_n)$. Therefore from Claim 3 of the proof of Lemma \ref{Bounded t_{+}}, which is valid in this case,  and using that in this case $\delta_0=1$,  we have 
\begin{equation}\label{EQ6NN*}
\frac{(W_n-|P_n(\bar{\beta}_n)|)^2}{|J_{\xi_n}(t_n+\alpha)|^2}\leq |J'_{\eta_n}(t_n+\bar{\beta}_n)|^2\leq \frac{(W_n+|P_n(\bar{\beta}_n)|)^2}{|J_{\xi_n}(t_n+\bar{\beta}_n)|^2},
\end{equation} 
where $P_n(\bar{\beta}_n):= J'_{\xi_n}(t_n+\bar{\beta}_n)\cdot J_{\eta_n}(t_n+\bar{\beta}_n)$.
From (\ref{EQNEW1}) and (\ref{EQ6NN*}) we have 
\begin{equation}\label{EQ6NN1*}
\frac{(W_n-|P_n(\bar{\beta}_n)|)^2\epsilon_n^2}{|J_{\xi_n}(t_n+\alpha)|^2}\leq \Bigg(\dfrac{\epsilon_n^2}{\beta_n}\Bigg )^2\leq \frac{(W_n+|P_n(\bar{\beta}_n)|)^2\epsilon_n^2}{|J_{\xi_n}(t_n+\bar{\beta}_n)|^2}.
\end{equation} 
Since $W_n$ converges to $W_0>0$ and $P_{n}(\bar{\beta}_n)$ converges to $0$, then from Claim 3 of the proof of Lemma \ref{Bounded t_{+}} and (\ref{EQ6NN1*}), we conclude that 
$$\ds\lim_{n\to \infty}\dfrac{\beta_n}{\epsilon_n^2}=\dfrac{1}{W_0},$$
as we wanted.
\end{proof}
Finally, we prove that Claim 8 (Claim 3 below) of the proof of Lemma \ref{Bounded t_{+}} is valid and consequently, we have the proof of lemma follow the same lines as the proof of Lemma \ref{Bounded t_{+}}. Remember we always keep the notion of last Lemma.\\
\ \\
\noindent \textbf{Claim 3:} For $n$ large enough $$\mu_n\geq \min\{a_0\epsilon_n^2, \sigma_n\}.$$ 
\begin{proof}[\emph{\textbf{Proof of Claim 3}}] From the definition of $\mu_n$ $J_{\eta_n}'(t_n-\mu_n)=0$. Assume by contradiction that $\mu_n<\min\{a_0\epsilon_n^2, \sigma_n\}$, then $\mu_n=\chi_n$, and from Claim 7 (of the proof of Lemma \ref{Bounded t_{+}}) $|J_{\eta_n}(t_n-\mu_n)|^2\leq \Gamma_0$. Therefore 
$$C^{-2}\lambda^{-2(t_n-\mu_n)}||\eta_n||\leq |J_{\eta_n}(t_n-\mu_n)|^2+|J'_{\eta_n}(t_n-\mu_n)\leq \Gamma_0.$$
Since $||\eta_n||$ is bounded and away from zero and $\mu_n<\min\{a_0\epsilon_n^2, \sigma_n\}$ is small, then the last inequality provides a contradiction.
\end{proof}
The remains of the proof of lemma follow the same lines as the proof of Lemma \ref{Bounded t_{+}}.
\end{proof}
%%%%%%%%%%%%%%%%%%%%%%%%%%%%%%%%%%%%%%%%%%%%%%%%%%%%%%%%%%%%%%%%%%%%%%%%%%%%
%%%%%%%%%%%%%%%%%%%%%%%%%%%%%%%%%%%%%%%%%%%%%%%%%%%%%%%%%%%%%%%%%%%%%%%%%%%%
\section{Proof of the Main Results}\label{Main-Results}
The main goal of this section is to prove Theorem \ref{main1} and Theorem \ref{main2}. Both proofs consist of proving that $\B^{s,u}=\emptyset$, and then using the Lemma \ref{Mane Lemma}.% we will complete the proof of Theorem \ref {main1} and Theorem \ref{main2}.
%The main goal of this section is to prove Theorem \ref{main1}, which will show that $\B^{s,u}=\emptyset$, so together with Lemma \ref{Mane Lemma} will conclude the the proof of Theorem \ref{main2}. 
%The main goal of this section is to prove the Theorem \ref{T-MAIN1}, which will be showed that $SM\setminus \B^{s,u}=SM$, which is equivalent to prove the Theorem \ref{main2}. % to prove that $ SM\setminus \B^{s,u}=SM$. The following theorem study the complement of set $\B^{s,u}$ and actually will be 
\ \\
\indent To avoid always mentioning the sectional curvature conditions,  we used the following definition:

\begin{Defi}
A complete manifold $M$ with bounded curvature below is \emph{nice} if: $M$ has dimension two or $M$ has dimension greater than two and its sectional curvature is bounded.
\end{Defi} 
This definition allows us to condense the Lemma \ref{Bounded t_{+}} and the Lemma \ref{Bounded t_{+}-dim 2} into the following lemma.
\begin{lemma}\label{Bounded-time}
If  $M$ is a nice manifold, then the sets $\mathcal{B}^{u}_{+}$ and $\mathcal{B}^{s}_{-}$ are closed.
\end{lemma}
Now, with the help of Lemma \ref{Bounded-time}, our task is reduced to proving the following Theorem.
% the following theorem, which implies the proof of Main Theorem.
\begin{T}\label{T-MAIN1} Let $M$ be a nice manifold with Anosov geodesic flow. Then, the sets 
 $$\Lambda^{s,u}=\Big\lbrace\theta\in SM: \text{for all} \,\, t\in \mathbb{R} \,\, \, \, \phi^{t}(\theta)\notin \B^{s,u}\Big \rbrace.$$ are closed, open, and nonempty  subsets of $SM$. 
\end{T}

Since $SM$ is a connected manifold, then any subset open, closed, and nonempty should be $SM$. Thus, as a corollary of Theorem \ref{T-MAIN1}, we have:
\begin{C}\label{Cor1-Main1}
If $M$ is nice, then  $\B^{s,u}=\emptyset$.% and, consequently, $M$ has no conjugate points.   
\end{C}

\noindent From the last Corollary \ref{Cor1-Main1}, we have that 
$$E^{s,u}(\theta)\cap V(\theta)=\emptyset, \,\,\, \text{for all}\,\,\, \theta\in SM.$$
Thus, since $E^{s,u}$ are Lagrangian, then by Lemma \ref{Mane Lemma} we can conclude that $M$ has no conjugate points, and \emph{afortiori} the proof of Theorem \ref{main1} and Theorem \ref{main2}. \

%%%%%%%%%%%%%%%%%%%%%%%%%%%%%%%%%%%%%%%%%%%%%%%%%%%%%%%%%%%%%%%%%%%%%%%%%%%%%%%
%%%%%%%%%%%%%%%%%%%%%%%%%%%%%%%%%%%%%%%%%%%%%%%%%%%%%%%%%%%%%%%%%%%%%%%%%%%%%%%
\begin{proof}[\emph{\textbf{Proof of Theorem \ref{T-MAIN1}}}] From Lemma \ref{Transfer Property}\,\,(Transfer Property) holds that $\Lambda^{s}=\Lambda^{u}$. Thus, it is sufficient to prove that $\Lambda^u$ is an open, closed, and nonempty set. \\% We {will} prove the unstable case, since {$\Lambda^{s}=\Lambda^{u}$}.\\
\ \\
\noindent \textbf{Nonempty:} It is a consequence of Lemma \ref{Nonempty-Good-Set}.\\
\ \\
\noindent \textbf{Closedness:}   It is an immediate consequence of Lemma \ref{Closed} (a similar argument was used by Ma\~n\'e in \cite{man:87}). \\
\ \\
\textbf{Openness:} We {will} prove that $SM \setminus \Lambda^{u}$ is closed. In fact, 
assume that $\theta_n \to \theta$ with $\theta_n\in SM\setminus \Lambda^{u}$, then there is $t_n\in \re$ such that $\phi^{t_n}(\theta_n)\in \B^{u}$.\\ 
\noindent  We have two cases to study. \\ 
\indent \textbf{Case 1:} For infinitely many indexes $n$, $t_n>0$. \\
In this case, for infinitely many indexes $n$,  $\theta_n\in \B^u_{+}$ and from Lemma \ref{Bounded-time} we have that $\theta\in \B^u_{+}$ and consequently $\theta \in SM\setminus \Lambda^{u}$. 
\\
\indent\textbf{Case 2:} For infinitely many indexes $n$, $t_n<0$. \\\
%In this case, without loss of generality, we can assume that 
%$$t_n=\max \Big\{t<0: \phi^{t}(\theta_n)\in \B^u \Big\}.$$ 
Then by Lemma \ref{Transfer Property}, we have that there is $\tilde{t}_n$  with $|t_n - \tilde{t}_n| \leq 1 + \sigma $ such that  $ \phi^{\tilde{t}_n}(\theta_n)\in \B^s$. 
\begin{itemize}
\item If for infinitely many indexes $\tilde{t}_{n}<0$, then $\theta_{n}\in \B^s_{-}$ and from Lemma \ref{Bounded-time} we have that $\theta \in \B^{s}_{-}$, and consequently $\theta\in SM\setminus \Lambda^s=SM\setminus \Lambda^u$.
\item If for infinitely many indexes $\tilde{t}_n > 0$, then $0< \tilde{t}_n  < 1 + \sigma+t_n< 1+\sigma$.
%\textcolor{red}{PAREI AQUI}%If $\tilde{t}_n < 0$ then there is $t_{-}^{s}(\theta_n)=\max \Big\{t<0: \phi^{t}(\theta_n)\in \B^s\Big\}$. Now define $r_n = \tilde{t}_n$ if $\tilde{t}_n  > 0$ and $r_n = t_{-}^{s}(\theta_n)$ if $\tilde{t}_n < 0$. From Lemma \ref{Bounded t^{s}_{-}}, we have
% $$    - \max\Big\{1 \, , \dfrac{\rho+ \log L^{u}(\theta)}{-2\log \lambda}\,  \Big\} \leq r_n \leq 1+ \sigma.                                                             $$
  Thus, passing to a subsequence if necessary,  from Lemma \ref{Le-Closed} we have that $\phi^{\tilde{t}_n}(\theta_n)\to \phi^{t_1}(\theta)\in \B^s$, which implies $\theta \in SM\setminus\Lambda^s=SM\setminus\Lambda^u$.\\%, by  Lemma \ref{Le-Closed},  that $\phi^{t_1}(\theta)\in \B^s$. Thus, by Lemma \ref{Transfer Property} there is $t_0$ such that $\phi^{t_0}(\theta)\in \B^u$ and then $\theta\in SM \setminus \Lambda^{u}$.\\
\end{itemize}
\vspace{-0.5cm}
%  \textbf{Nonempty:} %To conclude the proof of theorem, we  will prove that $\Lambda^{u}\neq \emptyset$. For this sake, observe that, s
% Since  $M$ is a non-compact manifold,  there {exists} a \emph{ray} $\gamma_{\theta}: [0,\infty)\to M$, \emph{i.e.}, $\gamma_{\theta}$ is a geodesic such that $d(\gamma_{\theta}(t), \gamma_{\theta}(s))=|t-s|$, which implies that $\gamma_{\theta}$ does not have {conjugate} points in $(0, +\infty)$. Therefore, from  Lemma \ref{NewLemma1*} we have $\theta \in \Lambda^u$ and the proof of the theorem is completed.
\end{proof}

%%%%%%%%%%%%%%%%%%%%%%%%%%%%%%%%%%%%
%%%%%%%%%%%%%%%%%%%%%%%%%%%%%%%%%%%%%%%%%%%%%%%%%%
%\bibliographystyle{alpha}	% (uses file "plain.bst")
%\bibliography{First_Paper}
\noindent \textbf{Sergio Augusto Roma\~na Ibarra}\\
Universidade Federal do Rio de Janeiro\\
Av. Athos da Silveira Ramos 149, Instituto de Matem\'atica, Centro de Tecnologia \ - Bloco C \ - Cidade Universit\'aria Ilha do Fund\~ao, cep 21941-909, Rio de Janeiro - Brasil\\
E-mail: sergiori@im.ufrj.br\\
\ \\
\noindent \textbf{\'Italo Dowell Lira Melo}\\
Universidade Federal do Piau\'i \\ 
Departamento de Matem\'atica-UFPI, Ininga, cep 64049-550 \\
Piau\'i-Brasil \\
E-mail: italodowell@ufpi.edu.br
\end{document}